\documentclass[12pt]{amsart} 

\usepackage{amsmath, amssymb, amsthm}

\usepackage{fullpage} 
\usepackage{thmtools, thm-restate}
\usepackage{braket}
\usepackage{eucal}
\usepackage[mathletters]{ucs}
\usepackage[utf8x]{inputenc}
\usepackage{tikz-cd}

\usepackage[all]{xy} 
\input xy

\xyoption{all}
\xyoption{2cell} 
\xyoption{v2}

\DeclareMathOperator{\sSet}{\mathbf{S}} 
\DeclareMathOperator{\ssSet}{\mathbf{SS}} 
\DeclareMathOperator{\Sets}{\mathbf{Set}} 
\DeclareMathOperator{\SCat}{\mathbf{SCat}} 
\DeclareMathOperator{\SGpd}{\mathbf{SGpd}} 
\DeclareMathOperator{\Kan}{\mathbf{Kan}} 
\DeclareMathOperator{\bL}{\mathbf{L}} 
\DeclareMathOperator{\bR}{\mathbf{R}} 

\DeclareMathOperator{\NN}{\mathbb{N}} 

\DeclareMathOperator{\sk}{\mathrm{sk}} 
\DeclareMathOperator{\Un}{\mathrm{Un}} 
\DeclareMathOperator{\St}{\mathrm{St}} 
\DeclareMathOperator{\ho}{\mathrm{ho}} 
\DeclareMathOperator{\cA}{\mathcal{A}}

\DeclareMathOperator{\cS}{\mathcal{S}}

\DeclareMathOperator{\map}{\mathrm{map}} 
\DeclareMathOperator{\Hom}{\mathrm{Hom}}
\renewcommand{\Re}{\mathrm{Re}}
\DeclareMathOperator{\Sing}{\mathrm{Sing}} 
 
\DeclareMathOperator{\Ob}{\mathrm{Ob}}
 
\DeclareMathOperator{\op}{\mathrm{op}}

\DeclareMathOperator{\sNerve}{\mathrm{N}_{\Delta}}

\theoremstyle{plain}
\newtheorem*{theorem*}{Theorem}
\newtheorem{theorem}{Theorem}[section]
\newtheorem{lemma}[theorem]{Lemma}
\newtheorem{proposition}[theorem]{Proposition}
\newtheorem{corollary}[theorem]{Corollary}

\theoremstyle{definition}
\newtheorem{definition}[theorem]{Definition}
\newtheorem{remark}[theorem]{Remark} 

\begin{document} 

\title[Covariant model structures and simplicial localization]{Covariant model structures and simplicial localization}
  \author[D.\ Stevenson]{Danny Stevenson}
  \address[Danny Stevenson]
  {School of Mathematical Sciences\\
  University of Adelaide\\
  Adelaide, SA 5005 \\
  Australia}
  \email{daniel.stevenson@adelaide.edu.au}

\thanks{This research was supported under the Australian
Research Council's {\sl Discovery Projects} funding scheme (project number DP120100106).}


\begin{abstract}
In this paper we prove that for any simplicial set $B$, there is a Quillen equivalence 
between the covariant model structure on $\mathbf{S}/B$ and a certain localization of the 
projective model structure on the category of simplicial presheaves on the simplex category 
$\Delta/B$ of $B$.  We extend this result to give a new Quillen equivalence between 
this covariant model structure and the projective model structure on the category of 
simplicial presheaves on the simplicial category $\mathfrak{C}[B]$.  
We study the relationship with Lurie's straightening theorem.  
Along the way we also prove some results on localizations of simplicial 
categories and quasi-categories.    
\end{abstract}
\maketitle

\tableofcontents 

\section{Introduction}

Let $B$ be a simplicial set and let $\Delta/B$ denote the simplex category of $B$.  
The starting point of this paper is the following observation: via Dugger's interpretation of the 
projective model structure on the category of simplicial presheaves 
$[(\Delta/B)^{\op},\sSet]$ as a {\em universal homotopy theory} \cite{Dugger}, 
left Kan extension along the canonical functor $y/B\colon \Delta/B\to \sSet/B$ induces a Quillen 
adjunction  
\[
\Re \colon [(\Delta/B)^{\op},\sSet]\rightleftarrows \sSet/B\colon \Sing 
\]
for the projective model structure on $[(\Delta/B)^{\op},\sSet]$ and the covariant 
model structure on $\sSet/B$, the category of simplicial sets over $B$.  Recall that the covariant model structure, 
due to Joyal and Lurie, has as its fibrant objects the {\em left fibrations} over 
$B$.  These are the quasi-categorical analogs of discrete left 
fibrations in ordinary category theory; just as discrete left fibrations correspond to 
diagrams of sets so too do left fibrations correspond to (homotopy coherent) diagrams of {\em spaces}.  
This is the content of the following theorem due to Lurie; the {\em straightening theorem}.  

\begin{theorem*}[Theorem 2.2.1.2 \cite{HTT}]
Let $B$ be a simplicial set.  Then there is a Quillen equivalence 
\[
\St_B\colon \sSet/B\rightleftarrows [\mathfrak{C}[B],\sSet] \colon \Un_B 
\]
where the category of simplicial functors $[\mathfrak{C}[B],\sSet]$ is equipped 
with the projective model structure and $\sSet/B$ is equipped with the covariant model structure.  
\end{theorem*}

Here $\mathfrak{C}[B]$ is the simplicial category of Example 1.1.5.8 of \cite{HTT}; it is 
also studied in \cite{DS,Riehl}.  The 
functors $\St_B$ and $\Un_B$ are called the {\em straightening} and {\em unstraightening} 
functors respectively.  An object of $[\mathfrak{C}[B],\sSet]$, i.e.\ a 
simplicial functor $F\colon \mathfrak{C}[B]\to \sSet$, can be thought of as 
an assignment of a simplicial set $F_b$ to every vertex of $B$, and a 
simplicial map $F_b\to F_{b'}$ to every edge from $b$ to $b'$, together with 
coherence data for higher dimensional simplices of $B$. Thus  
$F$ is a kind of homotopy coherent diagram of simplicial sets on $B$.  The effect of the functor 
$\St_B$ is to `straighten out' a space over $B$ into a homotopy coherent diagram of spaces.  

Recall that the functor $\mathfrak{C}[-]\colon \sSet\to \SCat$ forms part 
of an adjoint pair 
\[
\mathfrak{C}[-]\colon \sSet\rightleftarrows \SCat\colon \sNerve 
\]
where $\sNerve$ is the {\em simplicial} or {\em homotopy coherent nerve} 
functor (see Definition 1.1.5.5 of \cite{HTT} and also \cite{Cordier}).  In \cite{HTT} the straightening 
theorem is applied to prove a key theorem of Joyal and Lurie 
(see Theorem 1.1.5.13 of \cite{HTT}) which asserts that the 
adjoint pair $(\mathfrak{C}[-],\sNerve)$ is a Quillen adjunction 
for the Joyal model structure on $\sSet$ and the {\em Bergner model structure} \cite{Bergner} 
on $\SCat$.  The straightening theorem occupies a central place 
in the theory of quasi-categories (often called $\infty$-categories).

In this paper, amongst other things, we shall give a new proof 
of the straightening theorem.  Our approach will be to reduce to the special case in which $B$ is 
the nerve of a category via simplicial localization both for simplicial categories 
and quasi-categories.

To this end, our first main observation is that the covariant model structure on $\sSet/B$ can 
be obtained as a certain localization of the projective model structure on 
$[(\Delta/B)^{\op},\sSet]$.  We let $W\subset \Delta/B$  
denote the wide subcategory whose maps have as their underlying maps 
in the simplex category $\Delta$ the {\em initial vertex} maps, i.e.\ the maps 
$u\colon [m]\to [n]$ such that $u(0) = 0$.  
Under the Yoneda embedding $y_{\Delta/B}\colon \Delta/B\to [(\Delta/B)^{\op},\sSet]$, 
the set of arrows of $W$ is mapped to a set of arrows in $[(\Delta/B)^{\op},\sSet]$ that we shall 
also denote by $W$.  Thus we may consider the left Bousfield localization 
$L_W[(\Delta/B)^{\op},\sSet]$ of the projective model structure with respect to $W$.       
In Section~\ref{subsec:loc proj model structure} we prove the following result.

\begin{restatable}{theorem}{theoremA}
\label{thm:main}%
The Quillen adjunction $(\Re,\Sing)$ descends to a Quillen equivalence 
\[
L_{W}[(\Delta/B)^{\op},\sSet]\rightleftarrows \sSet/B 
\]
between the localized projective model structure 
and the covariant model structure 
on $\sSet/B$.  
\end{restatable}

This is the first of two Quillen equivalences that we obtain linking the covariant model 
structure on $\sSet/B$ with a localization of the projective model structure.  
The second of these Quillen equivalences arises as follows.  By composing with the forgetful functor 
$\sSet/B\to \sSet$, we may regard the 
functor $y/B\colon \Delta/B\to \sSet/B$ above as a simplicial diagram $y/B\colon \Delta/B\to \sSet$.  
The {\em simplicial replacement} of $y/B$ is then the bisimplicial set $s(y/B)$ whose $n$-th row is 
\[
s(y/B)_n = \bigsqcup_{\sigma_0\to \cdots \to \sigma_n} y/B(\sigma_0) 
\]
where the coproduct is taken over the set of $n$-simplices in the nerve of $\Delta/B$.  The 
bisimplicial set $s(y/B)$ comes equipped with a natural row augmentation $s(y/B)\to B$.  Given $X\in \sSet/B$ 
we may then form a bisimplicial set 
\[
s(y/B)\times_B X 
\]
which again has a natural row augmentation over $B$.  We may regard $s_!(X):= s(y/B)\times_B X$ as a 
simplicial presheaf on $\Delta/B$; this construction is functorial in $X$ and so defines a cocontinuous functor 
\[
s_!\colon \sSet/B\to [(\Delta/B)^{\op},\sSet]. 
\]
It follows that the functor $s_!$ forms part of an adjoint pair $(s_!,s^!)$.    
In Section~\ref{sec:presheaf straightening} we prove the following result.  

\begin{restatable}{theorem}{theoremB}
\label{thm:QE for tildeQ} 
Let $B$ be a simplicial set.  The adjoint pair 
\[
s_!\colon \sSet/B \rightleftarrows L_W[(\Delta/B)^{\op},\sSet] \colon s^! 
\]
is a Quillen equivalence for the 
covariant model structure on $\sSet/B$ and the localized projective model structure 
on $[(\Delta/B)^{\op},\sSet]$.     
\end{restatable}

To connect these theorems to the straightening theorem of Lurie, 
we need to relate the simplicial set $B$ to its simplex category $\Delta/B$.  
There is a well-known device which does this, namely the 
{\em last vertex} map.  This is a map $p_B\colon N(\Delta/B)\to B$  
defined as follows: since the domain and codomain of $p_B$ are cocontinuous functors 
of $B$, it suffices to define $p_B$ in the case where $B = \Delta[n]$ is a simplex.  In that 
case, $p_B$ is the nerve of the functor $\Delta/[n]\to [n]$ which sends $u\colon [m]\to [n]$ 
to $u(m)$.  We note the map $p_B$ has appeared in 
many places, see for instance Section 1.6 of \cite{Waldhausen}.     

This map $p_B$ is interesting in its own right; among other things it can be 
used to show that every simplicial set has the weak homotopy type of 
the nerve of a category.  We shall give a proof of an unpublished result 
of Joyal's (see 13.6 of \cite{Joyal-Notes}) asserting 
that the map $p_B$ exhibits the simplicial set $B$ as a {\em localization} 
of $N(\Delta/B)$ (this result also appears in \cite{TV2} in the context of Segal categories).  
To state the result, let us write $S$ for the set of {\em final vertex} 
maps in $\Delta/B$.  Thus a map $u\colon \Delta[m]\to \Delta[n]$ in 
$\Delta/B$ belongs to $S$ if and only if $u(m) = n$.  We then have   

\begin{restatable}[Joyal]{theorem}{flatteningthm} 
\label{thm:Joyals flattening thm}
Let $B$ be a simplicial set.  Then the canonical 
map $p_B\colon N(\Delta/B)\to B$ exhibits 
$B$ as a localization of $N(\Delta/B)$ with respect to the set $S$ of final 
vertex maps in $\Delta/B$.  In particular 
the induced map $L(N(\Delta/B),S)\to B$ 
is a weak categorical equivalence.  
\end{restatable} 

Here {\em localization} is understood in the context of quasi-categories 
(see Definition~\ref{def:quasi-localzn}).  
One may think of this theorem as an analog for simplicial sets of  
Theorem 2.5 from \cite{DK2}; the category $\Delta/B$ plays the role of the {\em flattening} 
of a simplicial category.  This theorem will play a key role for us;  
one important application of it is the following.  If $B$ is a simplicial set 
then we shall see in Section~\ref{subsec:reversed straightening} that there is a map 
\[
B^{\op}\to \sNerve\bL(B) 
\]
from the opposite simplicial set into the simplicial nerve of the simplicial category 
$\bL(B)$ of left fibrations over $B$.  When $B$ is a quasi-category this map 
sends a vertex $b$ of $B$ to the left fibration $B_{b/}$ over $B$, while an edge 
$f\colon a\to b$ is sent to a map $f_!\colon B_{b/}\to B_{a/}$ together with 
coherency data for higher dimensional simplices.  Of course the existence 
of such a map is well known, what we offer is a fresh perspective on how to 
construct it.    

The adjoint of this map is a simplicial functor 
$\mathfrak{C}[B]^{\op}\to \bL(B)$; we shall use it, together 
with Theorem~\ref{thm:main}, to give a simple proof of the 
following theorem.  

\begin{restatable}{theorem}{revstr}
\label{thm:reversed straightening} 
Let $B$ be a simplicial set.  The map $\mathfrak{C}[B]^{\op}\to \bL(B)$ induces a 
Quillen adjunction 
\[
\phi_!\colon [\mathfrak{C}[B],\sSet]\rightleftarrows \sSet/B\colon \phi^! 
\]
for the covariant model structure on $\sSet/B$ and the projective model 
structure on $[\mathfrak{C}[B],\sSet]$.  Moreover this Quillen adjunction 
is a Quillen equivalence.  
\end{restatable}      

We use this theorem, together with Theorem~\ref{thm:Joyals flattening thm} 
and Theorem 2.2 from \cite{DK2} to give a simple proof 
of the straightening theorem in Section~\ref{subsec:proof of straightening thm}.  
In fact, as we shall see (Remark~\ref{rem:simple straightening thm}), we do not need the full force of 
Theorem~\ref{thm:reversed straightening}, only a watered down version 
of it when $B = NC$ is the nerve of a category (this watered down version may 
be given a direct proof).

We shall also use Theorem~\ref{thm:Joyals flattening thm} to give a new model  
for the rigidification of a simplicial set $B$ into a simplicial category $\mathfrak{C}[B]$: 
we shall prove (Proposition~\ref{thm: comparison C[B] and deltaB}) that 
$\mathfrak{C}[B]$ is weakly equivalent in the homotopy category 
$\ho(\SCat)$ of the Bergner model structure to the hammock 
localization of $\Delta/B$ at a subcategory of initial or final 
vertex maps.

Along the way to proving all of these results, we prove some new results 
on simplicial localizations of categories, together with some new results on 
localizations of quasi-categories.

We comment on the relationship of our work to existing work by other authors.  Firstly, it goes without saying that 
this paper owes a tremendous debt to the foundational works of Joyal and Lurie, in particular the 
influence of Joyal's notes \cite{Joyal-Barcelona} and \cite{Joyal-Notes} will be clear.    
Secondly, after completing this work, we became aware of the very nice paper \cite{HM} which is closely related to 
our results.  For instance, the functor $s_!$ from our Theorem~\ref{thm:QE for tildeQ} is closely related to 
the rectification functor $r_!$ in Proposition B from \cite{HM}.  Our Theorems~\ref{thm:main} 
and~\ref{thm:QE for tildeQ} go beyond the work 
of \cite{HM} in that we allow more general base spaces than nerves of categories.  
Subsequent to the posting of the first version of the paper to the arXiv, 
the pre-print \cite{HM2} appeared which also gives a new proof of the straightening 
theorem as well as a result analogous to Theorem~\ref{thm:reversed straightening} 
above.  Our work complements \cite{HM2} in the following ways: firstly, our methods 
are quite different to those of \cite{HM2} and lead to what we believe is also a fairly conceptual 
proof of the straightening theorem; secondly, 
our work has the virtue of being self-contained, 
in particular we do not need to assume the full strength 
of the Quillen equivalence $(\mathfrak{C}[-],\sNerve)$ 
between the Joyal model structure and the 
Bergner model structure (we shall occasionally use the much 
more easily proved fact that $\mathfrak{C}[-]$ is left Quillen).  
The only result that we use but does not appear in 
\cite{HTT} is Proposition~\ref{prop:Riehl}; 
this is proven in \cite{Riehl} using the necklace technology of \cite{DS} (there 
is also an unpublished proof which does not depend on this technology).

We have 
tried to make the paper reasonably self contained, hopefully this does not make the paper even more 
tedious to read than it would be otherwise.  
We have assumed basic familiarity with standard material on 
quasi-categories and covariant model structures, for example 
the material in Chapters 1--2 of \cite{HTT} or the material in \cite{Joyal-Barcelona}.  We summarize 
the standard material on covariant model structures in Section~\ref{sec:cov model structure}, 
as well as providing proofs of some results which do not appear in these sources.  
In more detail, in Section~\ref{subsec:left anodyne maps} 
we recall some basic facts about left anodyne maps and prove Proposition~\ref{prop:contains left anodynes}, 
which gives a convenient way to determine whether a saturated class of monomorphisms in $\sSet$ 
contains the left anodyne maps.  In Section~\ref{subsec:cofinal} we prove Proposition~\ref{prop:joyal crit for cov eq} 
which gives a useful characterization (due to Joyal) of covariant equivalences; this in turn  
yields a useful characterization of cofinal maps (Theorem~\ref{thm:Joyals char of cov equiv}, due again to Joyal) 
which yields Quillen's Theorem A for quasi-categories as an easy corollary.   
We give a characterization (Proposition~\ref{prop:char of weak cat equivs}) of the class of those 
maps of simplicial sets which induce Quillen equivalences of covariant model structures by base change 
(this is a quasi-categorical analog of the main theorem from \cite{DK2}).  We also give an elementary 
proof of Theorem~\ref{thm:weak cat equivs and base change} 
(elementary in the sense that it only depends on the preceding standard facts about covariant model 
structures), asserting   
that weak categorical equivalences belong to this class.  

We begin Section~\ref{sec:bisimplicial sets} by recalling some basic facts about the category $\ssSet$ of 
bisimplicial sets, especially its structure as a simplicially enriched 
category.  In Section~\ref{sec:proj model structure}, for a fixed simplicial set $B$ we study the 
projective model structure on the slice category $\ssSet/B\Box 1$ 
and the horizontal Reedy model structure on $\ssSet/B\Box 1$ associated 
to the covariant model structure on $\sSet/B$.  We compare these two 
model structures in Proposition~\ref{prop:comparison}, using the realization and nerve functors 
$d$ and $d_*$.  

In Section~\ref{sec:left fibns of ssSet} we study the notion of a horizontal Reedy left fibration 
in $\ssSet$  and an enhanced version of this 
notion, the notion of a {\em strong} horizontal Reedy left fibration (Definition~\ref{def:horiz Reedy left fibn}).  
Associated to these notions is an allied concept of left anodyne map in $\ssSet$.  We study 
how these notions are related to left fibrations and left anodyne maps in $\sSet$ via the 
diagonal map $d\colon \ssSet\to \sSet$.  Our main result in this section is Theorem~\ref{thm:diag of Rezk lf}, which 
shows that the diagonal of a strong horizontal Reedy left fibration in $\ssSet$ is a left 
fibration in $\sSet$; a related result (Proposition~\ref{prop:d of col wise initial map}) 
shows that the diagonal of a level-wise cofinal 
map of bisimplicial sets is cofinal.  Theorem~\ref{thm:diag of Rezk lf} is a starting point 
for the development of a theory of covariant model structures for bisimplicial sets, 
which we will discuss in a future paper.  Theorem~\ref{thm:diag of Rezk lf} also plays a key role    
in Section~\ref{sec:main} where we establish the Quillen equivalences above, 
Theorem~\ref{thm:main} and Theorem~\ref{thm:QE for tildeQ} respectively.  

We turn our attention to localization of simplicial categories and quasi-categories in 
Section~\ref{sec:localization}.  We begin in Section~\ref{subsec:simp loc} by recalling a 
version of simplicial localization introduced by Lurie and show that this procedure 
gives simplicial categories DK-equivalent to the Dwyer-Kan simplicial localization introduced in 
\cite{DK1}.  In Section~\ref{subsec:qcat loc} we recall the definition of localization of 
quasi-categories due to Joyal and Lurie and relate this notion to Bousfield localizations 
of covariant model structures.  In Section~\ref{subsec:deloc} we prove Theorems~\ref{thm:Joyals flattening thm} 
and~\ref{thm: comparison C[B] and deltaB}, the latter giving the new model 
for the simplicial rigidification $\mathfrak{C}[B]$ mentioned previously.  
In Section~\ref{subsec:Lcofinal} we extend the notion of $L$-cofinal functor introduced in 
\cite{DK4} to the setting of quasi-categories and prove a generalization of Theorem (6.5) 
of \cite{DK4} (Theorem~\ref{thm:localization thm}).  
Finally, in Section~\ref{sec:straightening} we discuss the straightening theorem.

\medskip 

\noindent 
{\bf Notation}: With a few exceptions, we will use the notation from \cite{JT2} and 
\cite{Joyal-Barcelona}.  Thus we denote by $\sSet$ the category of simplicial sets.  The simplicial $n$-simplex in $\sSet$ 
is denoted by $\Delta[n]$.  We denote by $\emptyset$ the initial object of $\sSet$ and by 
$1$ the terminal object of $\sSet$ (i.e.\ the simplicial $0$-simplex $\Delta[0]$); sometimes we will 
denote both the category $[1]$ and the simplicial interval $\Delta[1]$ by $I$.  The groupoid completion of 
$[1]$ and its nerve will be denoted by $J$.  The fundamental category of a simplicial set 
$A$ will be denoted by $\tau_1(A)$.         

\section{The covariant model structure} 
\label{sec:cov model structure}
\subsection{The simplicial enrichment of $\sSet/B$}  
\label{subsec: simp enrichment of sSet/B}
Let $B$ be a simplicial set.  There is a canonical enrichment 
of the slice category $\sSet/B$ over the category of simplicial sets.  
If $\map(-,-)$ denotes the 
standard simplicial enrichment of $\sSet$, then for $X,Y\in \sSet/B$ 
the simplicial mapping space $\map_B(X,Y)\in \sSet$ 
is defined to be 
\[
\map_B(X,Y) = \map(X,Y)\times_{\map(X,B)} 1 
\]
where the map $1\to \map(X,B)$ is the structure map $X\to B$, regarded 
as a vertex of the simplicial set $\map(X,B)$.   
If $K\in \sSet$ and $X\in \sSet/B$ then the tensor $X\otimes K$ 
is defined by $X\otimes K = X\times K$, regarded as an object of $\sSet/B$ 
via the canonical map $X\times K\to X\to B$.  The cotensor $X^K$ 
is defined to be 
\[
X^K = \map(K,X)\times_{\map(K,B)} B, 
\]
where $B\to \map(K,B)$ is conjugate to the canonical map $B\times K\to B$.  
With these definitions we have the isomorphisms 
\[
\sSet/B(X\otimes K,Y) \simeq \sSet(K,\map_B(X,Y)) \simeq \sSet/B(X,Y^K), 
\]
natural in $X,Y\in \sSet/B$ and $K\in \sSet$.  

\subsection{The covariant model structure for simplicial sets} 
\label{subsec:cov model str}
In \cite{Joyal-Barcelona,HTT} it is proven that there is the structure of  
a simplicial model category on $\sSet/B$ for which the left fibrations over $B$ 
are the fibrant objects.  Recall that a map $p\colon X\to B$ is said to be a {\em left 
fibration} if it has the right lifting property (RLP) against the class of 
{\em left anodyne} maps in $\sSet$ (we review the concept of left anodyne map 
in more detail in the next section).  This model structure is called the {\em covariant} model structure on 
$\sSet/B$, it is described in the following theorem.    

\begin{theorem}[\cite{Joyal-Barcelona,HTT}] 
There is a structure of a left proper, combinatorial model category on 
$\sSet/B$ with respect to which a map $f\colon X\to Y$ in $\sSet/B$ is a 

\begin{itemize} 
\item cofibration if it is a monomorphism, 

\item weak equivalence if it is a covariant equivalence.  
\end{itemize} 

The fibrant objects for this model structure are precisely the left fibrations over $B$.  
The model structure is simplicial with respect to the simplicial enrichment above.  
\end{theorem} 

Recall that a map $f\colon X\to Y$ in $\sSet/B$ is a covariant equivalence if 
and only if the induced map on path components 
\[
\pi_0\map_B(f,Z)\colon \pi_0\map_B(Y,Z) \to \pi_0\map_B(X,Z) 
\]
is a bijection for every left fibration $Z\to B$ (see \cite{Joyal-Barcelona}).  
Following Joyal, let us write $\bL(B)$ for the full sub-category of $\sSet/B$ spanned 
by the left fibrations.  Note that $\bL(B)$ is enriched over $\Kan$, the category 
of Kan complexes. We will write $\Kan(B)$ for the full sub-category of $\sSet/B$ spanned by the 
Kan fibrations; note that $\Kan(B)\subset \bL(B)$.   

We note the following facts.  

\begin{theorem}[Joyal/Lurie] 
A map $X\to Y$ in $\bL(B)$ is a covariant equivalence if and only if it is a 
fiberwise weak homotopy equivalence in the sense that the induced map 
$X(b)\to Y(b)$ is a weak homotopy equivalence for all vertices $b$ of $B$.  
\end{theorem} 

Here $X(b)$ and $Y(b)$ denote the fibers of the maps $X\to B$ and $Y\to B$ 
at the vertex $b$; the next proposition shows that $X(b),Y(b)\in \Kan$.  

\begin{proposition}[Joyal/Lurie] 
If $B\in \Kan$ then $\bL(B) = \Kan(B)$.  
\end{proposition} 

Recall the following important examples of left fibrations: if $B$ is a quasi-category 
and $b\in B$ is a vertex, then the projection $B_{b/}\to B$ from the upper slice $B_{b/}$ 
is a left fibration (see for example Corollary 2.1.2.2 of \cite{HTT}).  
Similarly the projection $B^{b/}\to B$ from the fat upper slice $B^{b/}$ 
is a left fibration (see for example Proposition 4.2.1.6 of \cite{HTT}).  There is a canonical comparison 
map $B_{b/}\to B^{b/}$; it is a covariant equivalence (in fact a weak categorical equivalence --- 
see Proposition 4.2.1.5 of \cite{HTT}),  
for the identity arrow $1_b$ is a terminal object of both $B_{b/}$ and $B^{b/}$, 
from which it follows that $B_{b/}\to B^{b/}$ is right cofinal (Definition~\ref{def:right cof map}) and 
hence is a covariant equivalence.  In particular it follows (as in Corollary 4.2.1.8 of \cite{HTT}) that 
the mapping spaces $\mathrm{Hom}^L_B(a,b)$ and $\mathrm{Hom}^R_B(a,b)$ have the same homotopy type 
for all objects $a$ and $b$ of $B$.  Recall (see Section 1.2.2 of \cite{HTT}) that $\mathrm{Hom}^L_B(a,b)$ is the 
space of left morphisms from $a$ to $b$; it is defined as the fiber of the left fibration $B_{a/}\to B$ 
over the vertex $b$.  Similarly $\mathrm{Hom}^R_B(a,b)$ is the space of right morphisms from 
$a$ to $b$; it is defined as the fiber of the right fibration $B_{/b}\to B$ over the vertex $a$.     

Let us also note the following consequence of Proposition 2.1.2.5 of \cite{HTT}; if 
$B$ is a quasi-category and $f\colon a\to b$ is an edge in $B$, then the canonical map 
$B_{f/}\to B_{b/}$ is a trivial Kan fibration.  Choosing a section of this map 
and composing with the canonical map $B_{f/}\to B_{a/}$ we obtain a map 
\[
f_!\colon B_{b/}\to B_{a/} 
\]
in $\bL(B)$ which is well defined up to a contractible space of choices.  
If $\sigma\colon \Delta[2]\to B$ 
is a 2-simplex then the diagram 

\[
\begin{tikzcd} 
& \arrow[ld,"(d_2\sigma)_!"'] B_{\sigma(1)/} & \\
 B_{\sigma(0)/} & & \arrow[ll,"(d_1\sigma)_!"] 
B_{\sigma(2)/} \arrow[lu,"(d_0\sigma)_!"']  
\end{tikzcd} 
\]

in $\bL(B)$ commutes up to homotopy.  Later (Section~\ref{subsec:reversed straightening}) we shall 
see that in fact this construction extends to define a map $B^{\op}\to \sNerve\bL(B)$ 
into the homotopy coherent nerve of $\bL(B)$.

\subsection{Left anodyne maps in $\sSet$} 
\label{subsec:left anodyne maps}
In \cite{Joyal-Barcelona,HTT} it is shown that every left anodyne map in $\sSet/B$ is a covariant 
equivalence, where a map in $\sSet/B$ is said to be left anodyne if the underlying 
map of simplicial sets is so.  
Recall that a monomorphism  
in $\sSet$ is said to be {\em left anodyne} if it belongs to the saturated class of 
monomorphisms generated by the horn inclusions $\Lambda^k[n]\subset \Delta[n]$ 
for $0\leq k<n$, $n\geq 1$. 
For example, the {\em initial vertex} maps $\delta_n\colon \Delta[0]\to \Delta[n]$, defined by 
$\delta_n(0) = 0$, are left 
anodyne for every $n\geq 1$ (we sometimes denote these maps by 
$0\colon \Delta[0]\to \Delta[n]$).  More generally, we have the following result 
due to Joyal.   
\begin{proposition}[\cite{Joyal-Barcelona} Proposition 2.12]  
If $S\subset [n-1]$ is non-empty, then the generalized horn inclusion 
$\Lambda^S[n]\subset \Delta[n]$ is left anodyne.  
\end{proposition} 
Here if $S\subset [n-1]$ then the {\em generalized horn} $\Lambda^S[n]\subset \Delta[n]$ 
is defined to be 
\[
\Lambda^S[n] = \bigcup_{i\notin S} d^i\Delta[n-1],   
\]
where $d^i\colon \Delta[n-1]\to \Delta[n]$ denotes the inclusion of the $i$-th face.  

For later use we record the 
following extremely useful property of left anodyne morphisms.  
If $\cA$ is a class of monomorphisms in $\sSet$, then we say that 
$\cA$ satisfies the {\em right cancellation} property if whenever $u\colon A\to B$ and $v\colon B\to C$ are monomorphisms 
in $\sSet$ such that $vu,u\in \cA$, then $v\in \cA$.  

\begin{proposition}[\cite{Joyal-Barcelona} Corollary 8.15] 
\label{cancellation}
The class of left anodyne maps in $\sSet$ satisfies the right cancellation property.  
\end{proposition}

The next proposition gives a useful criterion to decide when a saturated 
class of monomorphisms in $\sSet$ contains the left anodyne morphisms.  
The proof that we give is based on the proof of Lemma 3.7 in \cite{JT2}, which 
gives a similar criterion for anodyne morphisms.    

\begin{proposition} 
\label{prop:contains left anodynes}
Let $\mathcal{A}$ be a saturated class of monomorphisms in $\sSet$ 
which satisfies the right cancellation property.  Then the following statements 
are equivalent: 
\begin{enumerate} 
\item $\mathcal{A}$ contains the class of left anodyne morphisms; 
\item $\mathcal{A}$ contains the initial vertex maps $\delta_n\colon \Delta[0]\to \Delta[n]$ for all $n\geq 1$; 
\item $\mathcal{A}$ contains the horn inclusions $h^0_n\colon \Lambda^0[n]\subset \Delta[n]$ for all $n\geq 1$.  
\end{enumerate} 
\end{proposition} 

\begin{proof} 
Clearly (1) implies (3).  We prove that (2) implies (1). 
We slavishly follow the ingenious strategy of Joyal and Tierney.  We prove that every horn 
$h^k_n\colon \Lambda^k[n]\to \Delta[n]$, $0\leq k<n$, belongs to $\mathcal{A}$.  
More generally, 
we prove by induction on $n\geq 1$ that every generalized horn $\Lambda^S[n]\subset \Delta[n]$ belongs 
to $\mathcal{A}$, where $S$ is a proper, non-empty subset of $[n]$ such that 
$n\notin S$.  Recall that 
\[
\Lambda^S[n] = \bigcup_{i\notin S} d^i\Delta[n-1].  
\]
The statement is true when $n=1$.  Assume the statement is true for $n>1$.  
Since $\delta_n = d^n\delta_{n-1}$, we see by the right cancellation property of $\mathcal{A}$ that  the 
map $d^n\Delta[n-1]\subset \Delta[n]$ belongs to $\mathcal{A}$ and so it therefore
suffices by the right cancellation property again to show that 
\[
d^n\Delta[n-1]\subset \Lambda^S[n] 
\]
belongs to $\mathcal{A}$.  

It suffices to show that $\Lambda^T[n]\subset \Lambda^S[n]$ belongs to $\mathcal{A}$ 
for any proper non-empty subsets $S\subset T\subset [n]$ such that $n\notin T$ (for then the 
desired statement follows by taking $T = [n-1]$).  Without loss of generality 
we may suppose that $T = S\cup \{i\}$ where $i\notin S$.  Let $U = (d^i)^{-1}T\subset [n-1]$.  A 
calculation, using the fact that the squares 
\[
\begin{tikzcd} 
\left[n-2\right] \arrow[r,"d^i"] \arrow[d,"d^{j-1}"'] & \left[n-1\right] \arrow[d,"d^j"] \\ 
\left[n-1\right] \arrow[r,"d^i"'] & \left[n\right] 
\end{tikzcd}
\]
are pullbacks for $i<j$, shows that $(d^i)^{-1}\Lambda^T[n] = \Lambda^U[n-1]$.  
The square 
\[
\begin{tikzcd} 
\Lambda^U\left[n\right] \arrow[r] \arrow[d] & \Lambda^T\left[n\right] \arrow[d] \\ 
\Delta\left[n-1\right] \arrow[r] & \Lambda^S\left[n\right] 
\end{tikzcd} 
\]
is a pushout and it therefore suffices to show that $\Lambda^U[n-1] 
\to \Delta[n-1]$ is in $\mathcal{A}$.  But $n-1\notin U$ since $n\notin T$ and $d^i(n-1) = n$; 
hence $\Lambda^U[n-1]\to \Delta[n-1]$ 
belongs to $\mathcal{A}$ by the induction hypothesis.  

We prove that (3) implies (2).  More generally, we prove by induction on 
$n\geq 1$ that every initial vertex map $\Delta[0]\to \Lambda^{\set{0,\ldots,i}}[n]$ belongs to 
$\cA$ for all $0\leq i<n$.  Then the desired statement follows by taking 
$i=0$ and composing with $h^0_n\colon \Lambda^0[n]\to \Delta[n]$.  
The statement is true when $n=1$.  By the inductive hypothesis the statement 
is true when $i=n-1$.  Since $\cA$ is closed under composition it suffices to show 
that $\Lambda^{\set{0,\ldots,i}}[n]\subset \Lambda^{\set{0,\ldots,i-1}}[n]$ belongs to 
$\cA$ for all $1\leq i<n$.  Arguing as above, the square 
\[
\begin{tikzcd} 
\Lambda^{\set{0,\ldots,i-1}}\left[n-1\right] \arrow[d] \arrow[r] & \Lambda^{\set{0,\ldots,i}}\left[n\right] \arrow[d] \\ 
\Delta\left[n-1\right] \arrow[r] & \Lambda^{\set{0,\ldots,i-1}}\left[n\right]
\end{tikzcd} 
\]
is a pushout.  By the inductive assumption and the right cancellation property, 
we see that the left-hand vertical map belongs to $\cA$, which 
implies that the right-hand vertical map belongs 
to $\cA$, as required.  
\end{proof} 

In fact, in the proof above it not really necessary that $\cA$ is 
saturated, all we have used is that $\cA$ contains all isomorphisms, 
and is closed under composition and forming pushouts along 
monomorphisms.  

\subsection{Covariant equivalences and cofinal maps} 
\label{subsec:cofinal} Most of the results that we state in this section can be found in the work of Joyal or Lurie, 
although proofs of some results have not yet been published.  An exception is 
Proposition~\ref{prop:char of weak cat equivs} which to the best 
of our knowledge is new (although the corresponding result for 
simplicial categories is well known \cite{DK2}).        

First let us recall the concept of a smooth map of simplicial sets (here we follow the terminology 
of Joyal in Definition 11.1 of \cite{Joyal-Barcelona} --- note that there is a disparity in the terminology 
of Joyal and Lurie regarding smooth maps).  A map $X\to B$ of simplicial sets 
is said to be {\em smooth} if the base-change functor $X\times_B(-)\colon \sSet/B\to \sSet/X$ sends 
left anodyne maps in $\sSet/B$ to left anodyne maps in $\sSet/X$.  Dually, 
$X\to B$ is said to be {\em proper} if $X\times_B(-)\colon \sSet/B\to \sSet/X$ sends 
right anodyne maps in $\sSet/B$ to right anodyne maps in $\sSet/X$.  It can be shown 
(Theorem 11.9 of \cite{Joyal-Barcelona} or Proposition 4.1.2.14 of \cite{HTT}) 
that every right fibration is smooth and that every left fibration is proper.  If 
$p\colon X\to B$ is a smooth map, then the adjoint pair 
\[
p^*\colon \sSet/B\rightleftarrows \sSet/X\colon p_* 
\]
is a Quillen adjunction for the covariant model structures on $\sSet/B$ and $\sSet/X$ 
(Theorem 11.2 of \cite{Joyal-Barcelona} or Proposition 4.1.2.7 of \cite{HTT}).

With these definitions understood, we have the following very useful characterization of 
covariant equivalences, due to Joyal (see \cite{Joyal-Notes}).  Let $\bR(B)$ denote the full subcategory of 
$\sSet/B$ spanned by the right fibrations on $B$.       

\begin{theorem}[Joyal] 
\label{prop:joyal crit for cov eq}
Let $B$ be a simplicial set.  The following statements are equivalent.  
\begin{enumerate} 
\item the map $X\to Y$ in $\sSet/B$ is a covariant equivalence; 
\item  the induced map 
$R\times_B X\to R\times_B Y$ is a weak homotopy equivalence for all $R\in \bR(B)$; 
\item the induced map $Rb\times_B X\to Rb\times_B Y$ is a weak homotopy equivalence for all 
$b\in B$, where $1\to Rb\to B$ is any factorization of $b\colon 1\to B$ into a right anodyne map 
$1\to Rb$, followed by a right fibration $Rb\to B$.
\end{enumerate}  
\end{theorem} 

We give a proof of this theorem below, since to the best of our knowledge one has not 
yet been given in the literature.  

\begin{proof} 
We prove that (1) implies (2).  Let $R\in \bR(B)$, then 
the composite functor $\sSet/B\to\sSet/R\to \sSet/1$ is left 
Quillen, where the functor $R\times_B (-)\colon \sSet/B\to \sSet/R$ is pullback 
along the smooth map $R\to B$, and where $\sSet/R\to \sSet/1$ is 
cobase change along the canonical map $R\to 1$.  Therefore, since a left Quillen functor preserves weak 
equivalences between cofibrant objects, and every object is cofibrant in the covariant model 
structure, it follows that $R\times_B X\to R\times_B Y$ is a covariant equivalence in $\sSet/1$.  But the 
covariant model structure on $\sSet/1$ coincides with the Quillen model structure on 
$\sSet$ and hence $R\times_B X\to R\times_BY$ is a weak homotopy equivalence. 

It is clear that (2) implies (3).  Therefore we need to prove that (3) implies (1).  
Suppose the hypotheses of (3) are satisfied; we need to show that $X\to Y$ is a covariant equivalence.  
There is a commutative diagram 
\[
\begin{tikzcd} 
X \arrow[r] \arrow[d] & Y \arrow[d] \\ 
X' \arrow[r] & Y' 
\end{tikzcd}
\]
in $\sSet/B$, where the maps $X\to X'$, $Y\to Y'$ are left anodyne and 
$X',Y'\in \bL(B)$.  By the two-out-of-three property it is sufficient to show that $X'\to Y'$ is a covariant 
equivalence and hence to show that $X'(b)\to Y'(b)$ is a weak homotopy equivalence for 
all vertices $b\in B$, where $X'(b)$ (respectively $Y'(b)$) denotes the fiber of $X'\to B$ 
(respectively $Y'\to B$) over $b\in B$.   

Let $b\in B$ be a vertex and factor the map $b\colon 1\to B$ as $1\to Rb\to B$ where 
$1\to Rb$ is right anodyne and $Rb\to B$ is a right fibration.  In the diagram 
\[
\begin{tikzcd} 
Rb\times_B X \arrow[r] \arrow[d] & Rb\times_B Y \arrow[d] \\ 
Rb\times_B X' \arrow[r] & Rb\times_B Y'
\end{tikzcd}
\]
the vertical maps are left anodyne (since $Rb\to B$ is smooth) and the map $Rb\times_B X\to Rb\times_B Y$ 
is a weak homotopy equivalence by hypothesis.  Therefore the map $Rb\times_B X'\to Rb\times_B Y'$ is 
a weak homotopy equivalence.  Consider the diagram 
\[
\begin{tikzcd} 
X'(b) \arrow[d] \arrow[r] &  Rb\times_B X' \arrow[d] \arrow[r] & X' \arrow[d] \\ 
1 \arrow[r] & Rb \arrow[r] & B
\end{tikzcd}
\]
in which both squares are pullbacks.
Since $Rb\times_B X'\to Rb$ is proper, the map $X'(b)\to Rb\times_B X'$ is right 
anodyne, and hence is a weak homotopy equivalence.  Analogously, $Y'(b)\to Rb\times_B Y'$ 
is a weak homotopy equivalence.  It follows that $X'(b)\to Y'(b)$ is a weak homotopy equivalence.    
\end{proof}

We now discuss cofinal maps of simplicial sets.  We use the following definition from \cite{HA}.   

\begin{definition} 
\label{def:right cof map}
A map $u\colon A\to B$ of simplicial sets is said to be {\em left cofinal} if it admits a 
factorization $u = pi$, where $i\colon A\to B'$ is right anodyne and $p\colon B'\to B$ is a 
trivial Kan fibration.  Dually, we say that $u\colon A\to B$ is {\em right cofinal} if the 
opposite map $u^{\op}\colon A^{\op}\to B^{\op}$ is left cofinal.  
\end{definition} 

Thus a map $u\colon A\to B$ of simplicial sets is left cofinal if it is 
{\em cofinal} in the sense of Definition 4.1.1.1 of \cite{HTT} (in the terminology of Joyal such a 
map is said to be {\em final}). Note that $u\colon A\to B$ is right cofinal if and only if it 
admits a factorization $u = pi$, where $i$ is left anodyne and $p$ is a trivial Kan fibration.  

For later use we record the statements of some elementary results 
on cofinal maps due to Joyal and Lurie.  

\begin{lemma}[Proposition 4.1.1.3 \cite{HTT}] 
\label{lem:right cofinal for mono}
A monomorphism $i\colon A\to B$ is right cofinal if and only if $i$ is left anodyne.  
\end{lemma}

\begin{lemma}[Joyal] 
The base change of a right cofinal map along a right fibration is right cofinal.  
\end{lemma} 

\begin{proof} 
This is clear, since left anodyne maps and trivial fibrations are preserved 
under base change by right fibrations.  
\end{proof} 

\begin{lemma}[Corollary 4.1.1.11 \cite{HTT}] 
\label{lem:left fibn + right cofinal => trivial Kan fibn} 
If $u\colon A\to B$ is a left fibration which is also right cofinal, then $u$ is a trivial Kan fibration.  
\end{lemma}

\begin{lemma}[Proposition 4.1.2.5 \cite{HTT}]  
\label{lem:cov equiv right cof}
A map $u\colon A\to B$ in $\sSet$ is right cofinal if and only if it is a covariant equivalence 
in $\sSet/B$.  
\end{lemma}

It follows easily from this result that right cofinal maps satisfy the right cancellation 
property; in other words if $u\colon A\to B$, $v\colon B\to C$ are maps such that $u$ 
and $vu$ are right cofinal, then so is $v$ (see Proposition 4.1.1.3 of \cite{HTT}).  

We shall also make use of the following related result: if $B$ is a simplicial set, then a map 
$X\to Y$ in $\bL(B)$ is a covariant equivalence if and only if it is right cofinal.  
For if such a map is a covariant equivalence then it clearly factors as a left 
anodyne map followed by a trivial Kan fibration.  The converse is clear.  

\begin{lemma} 
\label{lem:right cofinal stable under filt colimits}
Right cofinal maps in $\sSet$ are stable under filtered colimits.  
\end{lemma} 

\begin{proof} 
Suppose $u\colon A\to B$ is a filtered colimit of a family of right cofinal 
maps $u_\alpha\colon A_\alpha\to B_\alpha$.  Then each map $u_\alpha\colon 
A_\alpha\to B_\alpha$ is a covariant equivalence in $\sSet/B$.  Hence $u\colon A\to B$ 
is a covariant equivalence in $\sSet/B$, since covariant equivalences are stable 
under filtered colimits.  Factor $u\colon A\to B$ as $u=pi$ where $i\colon A\to B'$ 
is left anodyne and $p\colon B'\to B$ is a left fibration.  Then, arguing as in the proof 
of Lemma~\ref{lem:cov equiv right cof} above, we see that $p$ is a trivial Kan fibration.  
Hence $u\colon A\to B$ is right cofinal.     
\end{proof} 

The following theorem, due to Joyal (see 8.1 of \cite{Joyal-Notes}), gives a very useful criterion to recognize 
when a map of simplicial sets is cofinal.  

\begin{theorem}[Joyal] 
\label{thm:Joyals char of cov equiv}
The following statements are equivalent.  
\begin{enumerate} 
\item The map $u\colon A\to B$ is right cofinal; 
\item The induced map $R\times_B A\to R$ is a weak homotopy equivalence for all $R\in \bR(B)$; 
\item The simplicial set $Rb\times_B A$ is weakly contractible for every $b\in B$, 
where $1\to Rb\to B$ is any factorization of $b\colon 1\to B$ into a right anodyne map 
$1\to Rb$, followed by a right fibration $Rb\to B$.
\end{enumerate}
\end{theorem} 

\begin{proof} 
The theorem follows immediately from Lemma~\ref{lem:cov equiv right cof} and 
Theorem~\ref{prop:joyal crit for cov eq}. 
\end{proof} 

As a direct corollary of Theorem~\ref{thm:Joyals char of cov equiv}, we obtain the following key result: Quillen's Theorem A for 
quasi-categories, due to Joyal and Lurie.  
\begin{theorem}[Theorem 4.1.3.1 \cite{HTT}] 
\label{thm:A}
If $B$ is a quasi-category, then a map $u\colon A\to B$ is right cofinal if and only if the simplicial set 
$A\times_B B_{/b}$ is weakly contractible for every vertex $b\in B$.  
\end{theorem} 

The next result is obtained as an easy corollary of the Straightening Theorem in 
\cite{HTT} (see Remark 2.1.4.11 in \cite{HTT}).  In the absence of the latter theorem it is not so easy to prove.  
In \cite{HM} a proof is given by these authors using their Theorem C.  We 
give here another proof which only 
depends on the standard results in the theory of covariant and contravariant model 
structures presented in this section.  

\begin{theorem}[\cite{HTT}]
\label{thm:weak cat equivs and base change}
Let $f\colon A\to B$ be a weak categorical equivalence between simplicial sets.  
Then the Quillen adjunction 
\[
f_!\colon \sSet/A\rightleftarrows \sSet/B\colon f^*
\]
is a Quillen equivalence for the covariant model structures on $\sSet/A$ and $\sSet/B$.  
\end{theorem}

\begin{proof}
We begin by proving this theorem in the special case that the map $f$ is the inclusion 
$I_n\subset \Delta[n]$ of the $n$-{\em chain} 
\[
I_n := \Delta^{\set{0,1}}\cup 
\cdots \cup \Delta^{\set{n-1,n}}
\]
into $\Delta[n]$.  Let us write $i_n$ for this inclusion.  Suppose that $X\in \bL(\Delta[n])$ 
with structure map $p\colon X\to \Delta[n]$.  
We prove that the canonical map $I_n\times_{\Delta[n]}X\to X$ is left anodyne, from which it 
follows that $(i_n^*)^R$ is fully faithful.  Using Theorem~\ref{thm:Joyals char of cov equiv}, it suffices to prove   
that for every vertex $y\in X$, the canonical map $I_n\times_{\Delta[n]}X_{/y}\to X_{/y}$ 
is a weak homotopy equivalence.  In particular, it suffices to prove that this canonical map 
is right anodyne, since a right anodyne map is a weak homotopy equivalence.  Therefore, applying 
Theorem~\ref{thm:Joyals char of cov equiv} again, we see that it suffices to prove that the simplicial 
set $I_n\times_{\Delta[n]}X_{x//y}$ is weakly contractible, for every vertex $u\colon x\to y$ of $X_{/y}$, 
and where we have written $X_{x//y} = (X_{/y})_{u/}$ for ease of notation.  

We claim that the simplicial sets $\Delta^{\set{i}}\times_{\Delta[n]}X_{x//y}$ and $\Delta^{\set{i,i+1}}\times_{\Delta[n]} 
X_{x//y}$ are either empty (if $p(x)>i$ or $p(y)<i$ and $p(x)>i+1$ or $p(y)<i$ respectively) or contractible Kan complexes.  
This is enough to prove the claim that $I_n\times_{\Delta[n]}X_{x//y}$ is weakly contractible.  

The statement regarding the conditions under which each of these simplicial sets are empty is clear.  We prove 
that $\Delta^{\set{i,i+1}}\times_{\Delta[n]}X_{x//y}$ is a contractible Kan complex if $p(x)\leq i\leq p(y)$.  The proof 
that $\Delta^{\set{i}}\times_{\Delta[n]}X_{x//y}$ is a contractible Kan complex if $p(x)\leq i\leq p(y)$ is analogous and is 
left to the reader.  
We observe that $X_{x//y}\simeq (X_{\Delta^{\set{p(x),\ldots,p(y)}}})_{x//y}$ where if $S\subset \Delta[n]$ is a 
subcomplex then we write $X_S = S\times_{\Delta[n]}X$.  Thus we have an isomorphism 
\[
\Delta^{\set{i,i+1}}\times_{\Delta[n]} X_{x//y} \simeq 
\Delta^{\set{i,i+1}}\times_{\Delta^{\set{p(x),\ldots,p(y)}}}(X_{\Delta^{\set{p(x),\ldots,p(y)}}})_{x//y}.  
\]
It follows that we may suppose without loss of generality that $p(x) = 0$ and $p(y) = n$.  

We prove that any map $\partial\Delta[m]\to \Delta^{\set{i,i+1}}\times_{\Delta[n]}X_{x//y}$ extends 
along the inclusion $\partial\Delta[m]\subset \Delta[m]$.  
Given such a map, the induced map $\partial\Delta[m]\to \Delta^{\set{i,i+1}}$ factors 
through $\Delta[m]$ so that we have a commutative diagram 
\[
\begin{tikzcd} 
\partial\Delta[m] \arrow[r] \arrow[d] & \Delta^{\set{i,i+1}}\arrow[dd] \\
\Delta[m] \arrow[ur] \arrow[d] & \\  
\Delta[0]\star \Delta[m] \arrow[r] & \Delta[n]  
\end{tikzcd}
\]
The canonical map $\partial\Delta[m]\to X_{x//y}$ induces a map $\Lambda^0[m+1]\to X_{/y}$ 
which forms part of a commutative diagram 
\[
\begin{tikzcd} 
\Lambda^0[m+1] \arrow[d] \arrow[r] & X_{/y} \arrow[r] & X \arrow[d,"p"] \\ 
\Delta[m+1] \arrow[rr] & & \Delta[n]  
\end{tikzcd} 
\]
where $X_{/y}\to X$ is the canonical projection.  Since $p\colon X\to \Delta[n]$ is a left fibration we may choose  
a diagonal filler $\Delta[m+1]\to X$ for this diagram so that we have a commutative square 
\[
\begin{tikzcd} 
\Lambda^0[m+1] \arrow[d] \arrow[r] & X_{/y} \arrow[d] \\ 
\Delta[m+1] \arrow[r] & X 
\end{tikzcd} 
\]
Since $p(y) = n$ we have an isomorphism $X\simeq X\times_{\Delta[n]}\Delta[n]_{/p(y)}$ and so it follows that the right 
hand vertical map in this commutative square is a left fibration (Proposition 2.1.2.1 of \cite{HTT}).  Hence there exists 
a diagonal filler for this diagram which induces a map $\Delta[m]\to \Delta^{\set{i,i+1}}\times_{\Delta[n]}X_{x//y}$ 
extending the given map $\partial\Delta[m]\to \Delta^{\set{i,i+1}}\times_{\Delta[n]}X_{x//y}$.  This completes the proof 
that $I_n\times_{\Delta[n]}X\to X$ is left anodyne.

Suppose now that $X\to Y$ is a map in $\sSet/I_n$ such that $(i_n)_!X\to (i_n)_!Y$ is a 
covariant equivalence in $\sSet/\Delta[n]$.  We prove that $X\to Y$ is a covariant equivalence 
in $\sSet/I_n$.  By Theorem~\ref{prop:joyal crit for cov eq}, since $\Delta[n]_{/i} = \Delta[i]$,   
\[
(i_n)!X\times_{\Delta[n]}\Delta[i]\to (i_n)_!Y\times_{\Delta[n]}\Delta[i]
\]
is a weak homotopy equivalence for all $0\leq i\leq n$.  In other words 
\[
X\times_{I_n}i_n^*\Delta[i]\to Y\times_{I_n}i_n^*\Delta[i]
\]
is a weak homotopy equivalence for all $0\leq i\leq n$.  The map 
\[
\set{i}\to i_n^*\Delta[i] = \Delta^{\set{0,1}}\cup \cdots \cup \Delta^{\set{i-1,i}}
\] 
is clearly right anodyne for every $0\leq i\leq n$.  It follows 
from Theorem~\ref{prop:joyal crit for cov eq} again that  
$X\to Y$ is a covariant equivalence in $\sSet/I_n$.    
Hence $((i_n)_!,i_n^*)$ is a Quillen equivalence and so $(i_n^*)^R$ 
is essentially surjective.  

Next we show that $(u_!,u^*)$ is a Quillen equivalence for every inner-anodyne map 
$u\colon A\to B$.  To this end, let $\cA$ denote the class of monomorphisms $u\colon A\to B$ 
in $\sSet$ such that $u_!u^*X\to X$ is left anodyne for every $X\in \bL(B)$ and 
$(u^*)^R$ is essentially surjective.  It is straightforward to see 
that $\cA$ contains all isomorphisms, and is closed under retracts, coproducts and transfinite 
composition.  To show that it is also closed under pushouts requires a little more work. Suppose that 
\[
\begin{tikzcd}
A \arrow[d,"u"'] \arrow[r] & C \arrow[d,"v"] \\ 
B \arrow[r] & D 
\end{tikzcd}
\]
is a pushout diagram, where $u\in \cA$.  We first show that $v_!v^*X\to X$ 
is left anodyne.  Let $X\in \bL(D)$.  Then the diagram 
\[
\begin{tikzcd} 
X|_{A} \arrow[d] \arrow[r] & X|_{C} \arrow[d] \\ 
X|_{B} \arrow[r] & X 
\end{tikzcd} 
\]
is a pushout, and the left hand vertical map is left anodyne by hypothesis.  It follows that $v$ 
is also left anodyne.  To prove that $(v^*)^R$ is essentially surjective it suffices to show 
that if $X\in \bL(C)$, then we may find $Y\in \bL(B)$ such that 
$X|_A = u^*Y$.  For this we may adapt an argument of Joyal (see the proof of 
Lemma 2.2.4 of \cite{KLV}) replacing minimal Kan fibrations with 
minimal left fibrations; for the details we refer to the proof of 
Lemma 7.2 in \cite{HM}.  It follows that $\cA$ is saturated.  

We show that $\cA$ satisfies the right cancellation property.  
Suppose that $u\colon A\to B$ and $v\colon B\to C$ 
are monomorphisms such that $u$ and $vu$ belong to $\cA$.    
It is straightforward, using the right cancellation property for 
left anodyne maps, to see that $v_!v^*Z\to Z$ is left anodyne 
for all $Z\in \bL(C)$.  Let $X\in \bL(B)$.
Then there exists $Y\in \bL(C)$ together with a covariant equivalence 
$u^*v^*Y\to u^*X$ since $(u^*v^*)^R$ is essentially surjective.  
Note that $u^*v^*Y\to u^*X$ is right cofinal.  We may find a 
map $v^*Y\to X$ so that the 
diagram of simplicial sets
\[
\begin{tikzcd} 
u^*v^*Y\arrow[r] \arrow[d] & u^*X \arrow[d] \\ 
v^*Y\arrow[r] & X
\end{tikzcd} 
\]
commutes. Note that the vertical maps in this diagram are left anodyne.  It follows 
that $v^*Y\to X$ is right cofinal (right cofinal maps satisfy the right cancellation 
property --- see the remarks following 
Lemma~\ref{lem:cov equiv right cof}) and hence is a covariant equivalence in $\sSet/B$.  Therefore 
$(v^*)^R$ is essentially surjective.      

  Since 
$\cA$ contains the inclusions $I_n\subset \Delta[n]$ for all $n\geq 2$, it follows by 
Lemma 3.5 of \cite{JT2} that $\cA$ contains all the inner anodyne maps.  
It follows that $(u_!,u^*)$ is a Quillen equivalence for every inner anodyne map $u\colon A\to B$.  
A standard argument (see Lemma 7.1 of \cite{HM}) now shows that $(u_!,u^*)$ is a 
Quillen equivalence for every weak categorical equivalence $u\colon A\to B$ between 
simplicial sets $A$ and $B$.    
\end{proof}

A map $u\colon A\to B$ of simplicial sets is said to be fully faithful if 
a fibrant replacement $Ru\colon RA\to RB$ of $u$ is a fully faithful 
map between quasi-categories.  Likewise 
we say that a map $u\colon A\to B$ of simplicial sets is 
essentially surjective if $\tau_1(A)\to \tau_1(B)$ is essentially 
surjective.  In analogy with 1.3 (iii) of \cite{DK2}, 
let us say that a map $u\colon A\to B$ between simplicial sets is a 
{\em weak $r$-equivalence} if it is fully faithful, and if every object in 
$\tau_1(B)$ is a retract of an object in the image of $\tau_1(A)$. 
With this definition understood, we have the following analogue 
of Theorem 2.1 of \cite{DK2}.  

\begin{proposition}
\label{prop:char of weak cat equivs}
Let $u\colon A\to B$ be a map of simplicial sets.  Then 
$u$ is a weak $r$-equivalence if and only if the 
Quillen adjunction 
\[
u_!\colon \sSet/A\rightleftarrows \sSet/B\colon u^* 
\]
is a Quillen equivalence for the respective covariant 
model structures.  
\end{proposition}

Note that the dual version of this proposition, with the covariant model 
structure replaced by the contravariant model structure, is also true.  
Note also that this result is similar to the well known fact in ordinary 
category theory that idempotent completion does not change the category of 
presheaves.    

\begin{proof}
First observe that we may suppose without loss of generality that 
$A$ and $B$ are quasi-categories.  We shall show first that $(u_!)^L$ is fully faithful 
if and only if $u$ is fully faithful.  Let $a\in A$ 
be a vertex.  If $(u_!)^L$ is fully faithful then 
$A_{a/}\to A\times_B B_{u(a)/}$ is a covariant equivalence in 
$\bL(A)$, since $u_!A_{a/}\to B_{u(a)/}$ is left anodyne 
by the right cancellation property of left anodyne maps (Proposition~\ref{cancellation}).  
It follows easily that $u$ is fully faithful.  We leave the converse 
as an exercise for the reader.  Note that, using the fact that 
$\mathrm{Hom}^L_S(x,y)$ and $\Hom^R_S(x,y)$ are homotopy equivalent 
for any quasi-category $S$ (see the remarks above in Section~\ref{subsec:cov model str}), 
it follows that $u$ is fully faithful if and only if $(u_!)^L$ is fully 
faithful for the covariant model structures, if and only if 
$(u_!)^L$ is fully faithful for the contravariant model structures.    

Suppose now that $(u_!,u^*)$ is a Quillen equivalence.  We show 
that every object in $\tau_1(B)$ is a retract of an object in the image
of $\tau_1(A)$.  Let 
$b\in B$  be a vertex.  Since $(u^*)^R$ is fully faithful, the map 
$A\times_B B_{b/}\to B_{b/}$ is a covariant equivalence in $\sSet/B$.  
Therefore, by Theorem~\ref{prop:joyal crit for cov eq}, the simplicial set $A\times_B (B_{b/})_{/1_b}$ 
is weakly contractible; in particular it is non-empty.  Therefore there 
is a vertex $a\in A$ and a 2-simplex 
\[
\begin{tikzcd} 
& u(a) \arrow[dr,"g"] & \\ 
b \arrow[ur,"f"] \arrow[rr,"1_b"'] & & b 
\end{tikzcd} 
\]
in $B$.  This exhibits $b$ as a retract of $u(a)$ in $\tau_1(A)$.  

Suppose now that $u$ is fully faithful and every object of $\tau_1(B)$ 
is a retract of an object in the image of $\tau_1(A)$.  To show that 
$(u_!,u^*)$ is a Quillen equivalence we need to show that 
$(u^*)^R$ is fully faithful.  We will show that $A\times_B B_{/b}\to B_{/b}$ 
is left cofinal for every vertex $b\in B$.  This suffices to complete the proof 
for then $X\times_B A\times_B B_{/b}\to X\times_B B_{/b}$ is a weak 
homotopy equivalence for every vertex $b\in B$ and every $X\in \bL(B)$ 
by Theorem~\ref{thm:Joyals char of cov equiv}; hence $u_!u^*X\to X$ 
is a covariant equivalence in $\sSet/B$ by Theorem~\ref{prop:joyal crit for cov eq}.  

Choose a 2-simplex $\sigma\colon \Delta[2]\to B$ as above.  One shows easily   
that the diagram 
\[
\begin{tikzcd} 
& B_{/u(a)} \arrow[dr,"g_!"] & \\ 
B_{/b} \arrow[rr,"\mathrm{id}"'] \arrow[ur,"f_!"] & & B_{/b} 
\end{tikzcd} 
\]
in $\bR(B)$ commutes up to homotopy (compare with the remarks 
at the end of Section~\ref{subsec:cov model str}).  Choose a homotopy 
$H\colon \Delta[1]\times B_{/b}\to B_{/b}$ from $g_!f_!$ to $\mathrm{id}$.  Choose a factorization 
$B_{/u(a)}\to Y \to B_{/b}$ as a right anodyne map followed by a right fibration.  Since 
$\set{1}\times B_{/b}\to \Delta[1]\times B_{/b}$ is right anodyne we may choose a diagonal filler 
$\Delta[1]\times B_{/b}\to Y$ as indicated in the diagram  
\[
\begin{tikzcd} 
\set{1}\times B_{/b} \arrow[d] \arrow[r] & Y \arrow[d] \\ 
\Delta[1]\times B_{/b} \arrow[r] & B_{/b}.  
\end{tikzcd} 
\]
Then $B_{/u(a)}\to Y$ is a right anodyne map in $\bR(B)$ and the 
composite map $\set{0}\times B_{/b}\to \Delta[1]\times B_{/b}\to Y$ gives a retract 
diagram 
$B_{/b}\to Y\to B_{/b}$ in $\bR(B)$.  We obtain an induced diagram 
\[
\begin{tikzcd} 
A\times_B B_{/b} \arrow[d] \arrow[r] & A\times_B Y \arrow[d] \arrow[r] & A\times_B B_{/b} \arrow[d] \\ 
B_{/b} \arrow[r] & Y \arrow[r] & B_{/b} 
\end{tikzcd} 
\]
exhibiting $A\times_B B_{/b}\to B_{/b}$ as a retract of 
$A\times_B Y\to Y$.  Therefore it suffices to prove that 
$A\times_B Y\to Y$ is left cofinal.  The induced map 
$A\times_B B_{/u(a)}\to A\times_B Y$ is left cofinal, since it is 
a contravariant equivalence in $\bR(A)$, on account of the fact that 
$B_{/u(a)}\to Y$ is a right anodyne map in $\bR(B)$.  Similarly, the right 
cancellation property of left cofinal maps, and the fact 
that $(u_!)^L$ is fully faithful (for the {\em contravariant} model structure), shows that 
the canonical map $A\times_B B_{/u(a)}\to B_{/u(a)}$ is 
left cofinal.  The commutativity of the diagram 
\[
\begin{tikzcd} 
A\times_B B_{/u(a)} \arrow[d] \arrow[r] & A\times_B Y \arrow[d] \\ 
 B_{/u(a)} \arrow[r] & Y
\end{tikzcd} 
\]
and the right cancellation property of left cofinal maps shows that 
$A\times_B Y\to Y$ is left cofinal, which completes the proof.      
\end{proof}

\section{Bisimplicial sets} 
\label{sec:bisimplicial sets}

\subsection{The category of bisimplicial sets}
Let us write $\ssSet$ for the category of bisimplicial sets, that is, 
$\ssSet$ is the functor category $[\Delta^{\op}\times \Delta^{\op},\Sets]$.  If $X\in \ssSet$ then 
we say that $X_{m,n}:= X([m],[n])$ has {\em horizontal} degree $m$ and {\em vertical} 
degree $n$.  The {\em $m$-th column} of $X$ is the simplicial set $X_{m*}$ 
whose set of $n$-simplices is $(X_{m*})_n := X_{m,n}$.  The {\em $n$-th row} 
of $X$ is the simplicial set $X_{*n}$ whose set of $m$-simplices is 
$(X_{*n})_m := X_{m,n}$.   A map $X\to Y$ in $\ssSet$ is called a {\em row-wise} weak homotopy equivalence 
if the maps $X_{*n}\to Y_{*n}$ are weak homotopy equivalences for all $n\geq 0$;  
it is called a {\em column-wise} weak homotopy equivalence if the maps 
$X_{m*}\to Y_{m*}$ are weak homotopy equivalences for all $m\geq 0$.  
  
Recall that there is a canonical functor 
\[
(-)\Box(-)\colon \sSet\times \sSet\to \ssSet 
\]
which sends a pair of simplicial sets $K,L$ to their {\em box product} $K\Box L$.  
This is the bisimplicial set whose set of $(m,n)$-bisimplices is the set $K_m\times L_n$.  
Equivalently, $K\Box L = p_1^*K\times p_2^*L$, where $p_1^*,p_2^*\colon \sSet\to \ssSet$ 
are the functors induced by restriction along the two projections $p_1,p_2\colon \Delta\times \Delta\to \Delta$ 
respectively.  The bisimplicial set $\partial\Delta[m,n]$ is defined in terms of the box product as 
\[
\partial\Delta[m,n] = \partial\Delta[m]\Box \Delta[n]\cup \Delta[m]\Box \partial\Delta[n].  
\]
We recall the following fact from \cite{JT2}.  
\begin{proposition}[\cite{JT2}] 
The inclusions $\partial\Delta[m,n]\subset \Delta[m,n]$ for $m,n\geq0$ generate the monomorphisms 
in $\ssSet$ as a saturated class.  
\end{proposition}

Write $\delta\colon \Delta\to \Delta\times \Delta$ for the diagonal inclusion, so that 
$\delta([n]) = ([n],[n])$ for $[n]\in \Delta$.  Restriction along $\delta$ is the functor 
$d := \delta^*\colon \ssSet\to \sSet$ which sends a bisimplicial set $X$ to its 
{\em diagonal} $dX$.  Recall that $d$ has a left adjoint $\delta_!$ and a right adjoint 
$d_*:= \delta_*$.  The following result, due to Joyal and Tierney, is a direct 
consequence of Proposition B.0.17 of \cite{Joyal-Barcelona}.  

\begin{proposition}[Joyal-Tierney] 
\label{prop:delta!(mono)}
The functor $\delta_!\colon \sSet\to \ssSet$ sends monomorphisms of simplicial sets to monomorphisms 
of bisimplicial sets.  
\end{proposition} 

As an immediate corollary, using the fact that the Reedy model structure and injective model 
structure coincide on bisimplicial sets, we have the following useful result.  

\begin{corollary} 
\label{corr:d sends triv fibns to triv fibns}
The diagonal functor $d\colon \ssSet\to \sSet$ sends Reedy trivial fibrations in 
$\ssSet$ to trivial Kan fibrations in $\sSet$.  
\end{corollary} 

Finally, let us recall (see for instance \cite{JT2}) 
that for fixed simplicial sets $K$, $L$, the functors $K\Box (-)\colon \sSet
\to \ssSet$ and $(-)\Box L\colon \sSet\to \ssSet$ have right adjoints $K\backslash (-)\colon 
\ssSet\to \sSet$ and $(-)/L\colon \ssSet\to \sSet$ respectively.

\subsection{Simplicial enrichments} 

A {\em simplicial space} is a simplicial object in $\sSet$.  Therefore, 
each bisimplicial set $X$ may be regarded as a simplicial space in two different ways.  
On the one hand we may regard $X$ as a 
{\em horizontal} simplicial object in $\sSet$ whose $m$-th {\em column} is the simplicial set $X_{m*}$.  
On the other hand we may regard $X$ as a 
{\em vertical} simplicial object in $\sSet$ whose $n$-th {\em row} is the simplicial set $X_{*n}$.  

Corresponding to the canonical simplicial enrichment of $\sSet$, 
there are two simplicial enrichments of $\ssSet$ depending on whether 
we view bisimplicial sets as horizontal or vertical simplicial objects in $\sSet$.  

Fortunately, in this paper we will only have need to consider one of these 
simplicial enrichments, the {\em horizontal} simplicial enrichment, which is the natural 
enrichment when bisimplicial sets are viewed as horizontal simplicial spaces.  The tensor 
for this enrichment is defined to be $X\otimes K = X\times p_2^*K$ for $X\in \ssSet$ 
and $K\in \sSet$.  The simplicial mapping space is defined by the formula 
\[
\map(X,Y) = (p_2)_*Y^X, 
\]
for $X,Y\in \ssSet$, where $Y^X$ denotes the exponential in the cartesian closed 
category $\ssSet$ and $(p_2)_*\colon \ssSet\to \sSet$ is the functor which sends a bisimplicial 
set $Z$ to its first column $Z_{0*}$.  If $X\in \ssSet$ and $K\in \sSet$ then we write 
$X^K$ for their cotensor.

The proof of the following lemma is straightforward and is left to the reader.  
\begin{lemma} 
\label{lem:slice simp adjunction}
Let $K$ be a simplicial set.  Then there is an isomorphism of simplicial sets 
\[
\map(K\Box L,X)\simeq \map(L,K\backslash Y),  
\]
natural in $L\in \sSet$ and $Y\in \ssSet$.  In other words the adjunction $K\Box (-)\dashv K\backslash (-)$ 
above is a simplicial adjunction for the horizontal simplicial enrichment.  
\end{lemma}     

Suppose now that 
$B\in \ssSet$.  Then the horizontal simplicial enrichment on $\ssSet$
induces a canonical simplicial enrichment on the slice category 
$\ssSet/B$, which we will sometimes refer to as the horizontal simplicial 
enrichment on $\ssSet/B$.  

If $X\in \ssSet/B$ and $K\in \sSet$, then the tensor $X\otimes K$ 
can be naturally regarded as an object of $\ssSet/B$ via the canonical 
map $X\otimes K\to X\to B$.  If $X,Y\in \ssSet/B$, then the 
simplicial mapping space, $\map_B(X,Y)$, is defined to be the fiber
\[
\map_B(X,Y) = \map(X,Y)\times_{\map(X,B)} 1 
\]
where the canonical map $X\to B$ is regarded as a vertex $1\to \map(X,B)$ 
of $\map(X,B)$.  Similarly we define the cotensor, $\map_B(K,X)$, for 
$X\in \ssSet/B$, $K\in \sSet$ to be 
\[
\map_B(K,X) = X^K\times_{B^K} B, 
\]
where $B\to B^K$ is the conjugate of the 
canonical map $B\otimes K\to B$.  With these definitions we have the 
sequence of isomorphisms 
\[
\ssSet/B(X\otimes K,Y) \simeq \sSet(K,\map_B(X,Y)) \simeq \ssSet/B(X,\map_B(K,Y)) 
\]
natural in $X,Y\in \ssSet/B$ and $K\in \sSet$.

\subsection{The projective model structure on $\ssSet$} 
\label{sec:proj model structure}
Let us identify the category $\ssSet$ with the category $s(\sSet)$ 
of horizontal simplicial spaces, so that $\ssSet = s(\sSet)$ 
is simplicially enriched with respect to the horizontal simplicial 
enrichment.  

Recall that the (horizontal) {\em projective model structure} on $\ssSet$ 
has as its weak equivalences the maps $X\to Y$ in $\ssSet$ which are 
column-wise weak homotopy equivalences, in other words $X_{m*}\to Y_{m*}$ 
is a weak homotopy equivalence for all $m\geq 0$.  The fibrations in this 
model structure are the column-wise Kan fibrations $X\to Y$ in $\ssSet$, i.e.\ 
$X_{m*}\to Y_{m*}$ is a Kan fibration for all $m\geq 0$.  Recall that the horizontal 
projective model structure is simplicial for the horizontal simplicial enrichment on 
$\ssSet$.  

Let $B$ be a simplicial set.  Then the (horizontal) projective model structure on $\ssSet$ 
induces a model structure on the slice category $\ssSet/B\Box 1$.  When 
there is no danger of confusion we will refer to this overcategory model structure 
as the (horizontal) projective model structure on $\ssSet/B\Box 1$.  The weak 
equivalences (respectively fibrations) in this model structure are the maps $X\to Y$ in $\ssSet/B\Box 1$ 
such that the underlying map is a column-wise weak homotopy equivalence 
(respectively Kan fibration) in $\ssSet$.  
The projective model structure on 
$\ssSet/B\Box 1$ is simplicial with respect to the simplicial enrichment on 
$\ssSet/B\Box 1$ induced by the horizontal simplicial enrichment on 
$\ssSet$.

The category $s(\sSet/B)$ of (vertical) simplicial 
objects in $\sSet/B$ may be identified with the category $\ssSet/B\Box 1$.  There is a 
canonical pair of adjoint functors   
\[
d\colon \ssSet/B\Box 1\rightleftarrows \sSet/B\colon d_* 
\]
defined as follows.  The functor $d$ sends an object $X\in \ssSet/B\Box 1$ to 
its diagonal simplicial set $dX$.  Note that since $d(K\Box L) = K\times L$ for 
all $K,L\in \sSet$, the simplicial set $dX$ is equipped with a canonical map 
$dX\to B$ and hence can be regarded as an object of $\sSet/B$.  The functor $d_*\colon \sSet/B\to 
\ssSet/B\Box 1$ is the functor which sends an object $X$ of $\sSet/B$ to the 
vertical simplicial object of $s\sSet/B$ whose $n$-th row is  
\[
(d_*X)_n  = (d_*X)_{*n}    := X^{\Delta[n]}.  
\]
Note that the functors $d$ and $d_*$ are simplicial with respect to the horizontal 
simplicial enrichment of $\ssSet/B\Box 1$.  
The following lemma is well known (see for instance VII Lemma 3.4 of \cite{GJ}).  
\begin{lemma} 
Let $X\in \bL(B)$.  Then $d_*X$ is Reedy fibrant for the 
(vertical) Reedy model structure on $s(\sSet/B) = \ssSet/B\Box 1$ associated to 
the covariant model structure on $\sSet/B$.  
\end{lemma} 

Alternatively, we may consider $d_*X$ as a horizontal simplicial space.  
\begin{lemma} 
\label{lem:d_*X is W local}
Let $X\in \bL(B)$.  Then $d_*X$ is projectively fibrant for the horizontal 
projective model structure on $\ssSet/B\Box 1$.  Moreover, the maps 
\[
(d_*X)_{m*}\to (d_*X)_{0*}\times_{B_0} B_m 
\]
induced by the initial vertex maps $\delta_m\colon \Delta[0]\to \Delta[m]$ are weak homotopy equivalences 
for all $m\geq 0$.  
\end{lemma} 

\begin{proof} 
We prove the first statement.  Recall that $d_*X$ is the vertical simplicial object in $\sSet/B$ whose $n$-th row  
is $X^{\Delta[n]}\in \sSet/B$.  It follows then that the $m$-th column $(d_*X)_{m*}$ 
fits into a pullback diagram 
\[
\begin{tikzcd} 
(d_*X)_{m*} \arrow[d] \arrow[r] & \map(\Delta[m],X) \arrow[d] \\ 
B_m \arrow[r] & \map(\Delta[m],B) 
\end{tikzcd} 
\]
where the map $B_m\to \map(\Delta[m],B)$ is the inclusion of the set 
of vertices of $\map(\Delta[m],B)$.  
Since $\map(\Delta[m],X)\to \map(\Delta[m],B)$ is a left fibration it 
follows that $(d_*X)_{m*}\to B_m$ is a Kan fibration, as $B_m$ is discrete.  
Therefore $d_*X$ is fibrant in the horizontal projective model structure 
on $\ssSet/B\Box 1$.  

We prove the second statement.  It follows from the discussion in the preceding paragraph that the 
maps in question are the maps 
\[
\map(\Delta[m],X)\times_{\map(\Delta[m],B)}B_m\to 
X\times_B B_m, 
\]
induced by the initial vertex maps $\delta_m\colon \Delta[0]\to \Delta[m]$.  Thus these maps are pullbacks 
of the maps 
\[
\map(\Delta[m],X)\to X\times_B \map(\Delta[m],B) 
\]
induced by the initial vertex maps $\delta_m\colon \Delta[0]\to \Delta[m]$.  Since these initial vertex maps 
are left anodyne and $X\in \bL(B)$, it follows that the latter maps are trivial Kan fibrations 
(Corollary 2.1.2.9 of \cite{HTT}).  Hence the maps above are trivial Kan fibrations, in particular they are 
weak homotopy equivalences. 
\end{proof} 

\begin{proposition} 
\label{prop:comparison}
The adjunction $d\colon \ssSet/B\Box 1\rightleftarrows \sSet/B\colon d_*$ is a Quillen adjunction 
for the horizontal projective model structure on $\ssSet/B\Box 1$ and the covariant model structure 
on $\sSet/B$.  
\end{proposition} 

\begin{proof} 
It is sufficient to show that $d_*\colon \sSet/B\to \ssSet/B\Box 1$ sends fibrations between fibrant 
objects to projective fibrations, and trivial fibrations to trivial projective fibrations.  Suppose $X\to Y$ 
is a covariant fibration in $\bL(B)$, so that $X\to Y$ is a left fibration.  From the proof of 
the previous lemma it follows that $(d_*X)_{m*}\to (d_*Y)_{m*}$ is a left fibration in $\bL(B_m)$ for every $m\geq 0$.  
Hence it is a left fibration between Kan complexes and hence is a Kan fibration.  
The proof that $d_*$ sends trivial fibrations in $\sSet/B$ to trivial projective fibrations is analogous.  
\end{proof}
 
\begin{lemma} 
\label{lem:diags of row-wise cov eqs are cov eqs}
Suppose that $X\to Y$ is a map in $\ssSet/B\Box 1$ such that $X_{*n}\to Y_{*n}$ 
is a covariant equivalence in $\sSet/B$ for all $n\geq 0$.  Then $dX\to dY$ is a 
covariant equivalence in $\sSet/B$.  
\end{lemma} 

\begin{proof} 
It is sufficient to show that every object of $\ssSet/B\Box 1$ is Reedy cofibrant 
for the (vertical) Reedy model structure on $\ssSet/B\Box 1 = s(\sSet/B)$ associated 
to the covariant model structure on $\sSet/B$.  
Let $\cA$ be the class of monomorphisms in $\ssSet/B\Box 1$ which have the 
left lifting property (LLP) against all Reedy trivial fibrations.  Then $\cA$ is saturated, and clearly contains 
all of the canonical inclusions $\partial\Delta[m,n]\subset \Delta[m,n]$ 
in $\ssSet/B\Box 1$.  But then $\cA$ contains all monomorphisms in $\ssSet/B\Box 1$ 
by Proposition 2.2 from \cite{JT2}.  
%
\end{proof}

\subsection{Left fibrations of bisimplicial sets} 
\label{sec:left fibns of ssSet}
\begin{definition} 
\label{def:horiz Reedy left fibn}
A map $p\colon X\to Y$ in $\ssSet$ is said to be a horizontal {\em Reedy left fibration} 
if it has the RLP against all maps of the form 
\[
\partial\Delta[m]\Box \Delta[n]\cup \Delta[m]\Box \Lambda^k[n]\subset \Delta[m,n] 
\]
for $m\geq 0$, $0\leq k<n$, $n\geq 1$.  A horizontal Reedy left fibration $p\colon X\to Y$ 
is said to be {\em strong} if in addition the maps 
\[
X_{n*}\to X_{0*}\times_{Y_{0*}} Y_{n*} 
\]
induced by $0\colon \Delta[0]\to \Delta[n]$ are trivial Kan fibrations 
for every $n\geq 0$.    
\end{definition} 

The following observations are clear.  

\begin{lemma} 
\label{lem:stable under bc}
Horizontal Reedy left fibrations and strong horizontal Reedy left fibrations are stable under base change. 
\end{lemma} 

\begin{lemma} 
\label{lem:left slice of horiz lf is lf}
Let $X\to Y$ be a horizontal Reedy left fibration in $\ssSet$ and let $A\to B$ 
be a monomorphism in $\sSet$.  Then the induced map 
\[
B\backslash X \to A\backslash X\times_{A\backslash Y} B\backslash Y 
\]
is a left fibration.  
\end{lemma} 

In particular $X_{m*}\to Y_{m*}$ is a left fibration for all $m\geq 0$.  

\begin{lemma} 
Let $X\to Y$ be a horizontal Reedy left fibration in $\ssSet$ and let $A\to B$ be a 
left anodyne map in $\sSet$.  Then the induced map 
\[
X/B \to X/A\times_{Y/A} Y/B 
\]
is a trivial Kan fibration.  
\end{lemma} 

The following lemma gives a useful way to recognize strong horizontal Reedy left fibrations.  

\begin{lemma} 
\label{lem:recognize str lf}
Let $B$ be a simplicial set.  Suppose that $X\to B\Box 1$ is a horizontal Reedy left fibration 
such that the maps 
\[
X_{n*}\to X_{0*}\times_{B_0} B_n 
\]
induced by the initial vertex maps $0\colon [0]\to [n]$ are weak homotopy equivalences 
for all $n\geq 0$.  Then $X\to B\Box 1$ is a strong horizontal Reedy left fibration.  
\end{lemma} 

\begin{proof} 
We need to show that the induced maps $X_{n*}\to X_{0*}\times_{B_0} B_n$ are trivial 
Kan fibrations for all $n\geq 0$.  But $X\to B\Box 1$ 
is a horizontal Reedy left fibration, and hence the map above 
is a left fibration (Lemma~\ref{lem:left slice of horiz lf is lf}).  
Now $X_{0*}\to B_0$ is also a left fibration (Lemma~\ref{lem:left slice of horiz lf is lf} again), and hence $X_{0*}$ is a Kan 
complex since $B_0$ is discrete.  Therefore $B_n\times_{B_0}X_{0*}$ 
is a Kan complex and hence $X_{n*}\to B_n\times_{B_0} X_{0*}$ is a 
Kan fibration.  Therefore it is a trivial Kan fibration; in other words, 
$X\to B\Box 1$ is a strong horizontal Reedy left fibration.   
\end{proof} 
 
\begin{definition} 
A monomorphism $i\colon A\to B$ in $\ssSet$ is said to be (horizontal) 
{\em left anodyne} if it belongs to the saturated class of monomorphisms 
generated by the maps of the form 
\[
\partial\Delta[m]\Box \Delta[n]\cup \Delta[m]\Box \Lambda^k[n]\subset \Delta[m,n] 
\]
for $m\geq 0$, $0\leq k<n$, $n\geq 1$.  
\end{definition} 

\begin{definition} 
We will say that a map $i\colon A\to B$ in $\ssSet$ is {\em column-wise left anodyne} 
if $i_m\colon A_{m*}\to B_{m*}$ is a left anodyne map in $\sSet$ for all $m\geq 0$.  
\end{definition} 

It is easy to see that if $K\to L$ is a left anodyne map in $\sSet$ then the induced map 
$J\Box K\to J\Box L$ in $\ssSet$ is column-wise left anodyne for any simplicial set $J$.  
Observe that column-wise left anodyne maps in $\ssSet$ form a saturated class of 
monomorphisms.    

\begin{lemma} 
\label{lem:horiz la implies col-wise la}
Suppose that $i\colon A\to B$ is a horizontal left anodyne map in $\ssSet$.  
Then $i$ is column-wise left anodyne.  
\end{lemma} 

\begin{proof} 
For any $m\geq 0$ and $0\leq k<n$, $n\geq 1$, the maps $\partial\Delta[m]\Box \Lambda^k[n] 
\to \partial\Delta[m]\Box \Delta[n]$ and $\Delta[m]\Box \Lambda^k[n] \to 
\Delta[m,n]$ are both column-wise left anodyne maps.  Therefore 
the canonical map 
\[
\Delta[m]\Box \Lambda^k[n] \to \partial\Delta[m]\Box \Delta[n] \cup \Delta[m]\Box \Lambda^k[n] 
\]
is a column-wise left anodyne map, since it is the pushout of one.  But now in the composite 
\[
\Delta[m]\Box \Lambda^k[n] \to \partial\Delta[m]\Box \Delta[n] \cup \Delta[m]\Box \Lambda^k[n] 
\to \Delta[m,n] 
\]
the first map and the composite are column-wise left anodyne maps.  Therefore, by the 
right cancellation property of left anodyne maps in $\sSet$ (Proposition~\ref{cancellation}) it follows 
that $\partial\Delta[m]\Box \Delta[n] \cup \Delta[m]\Box \Lambda^k[n] \to \Delta[m,n]$ is 
column-wise left anodyne.  

To complete the proof, observe that the functor $\Delta[m]\setminus -\colon \ssSet\to \sSet$ which 
sends a bisimplicial set $X$ to its $m$-th column $X_{m*}$ has a right adjoint and hence sends 
saturated classes to saturated classes.   
\end{proof} 

It follows that the class of horizontal left anodyne maps in $\ssSet$ is equal 
to the class of column-wise left anodyne maps, but we will not need this.  

\begin{lemma} 
\label{lem:diag of la map}
Suppose that $i\colon A\to B$ is a horizontal left anodyne map in $\ssSet$.  Then 
the diagonal $di\colon dA\to dB$ is a left anodyne map in $\sSet$.  
\end{lemma} 

\begin{proof} 
It suffices to prove the statement when $i$ is the inclusion $\partial\Delta[m]\Box \Delta[n]\cup 
\Delta[m]\Box \Lambda^k[n]\subset \Delta[m,n]$ with $0\leq k<n$.  But then $di$ is the map 
\[
\partial\Delta[m]\times \Delta[n]\cup \Delta[m]\times \Lambda^k[n]\subset \Delta[m]\times \Delta[n]
\]
which is left anodyne by Corollary 2.1.2.7 of \cite{HTT} or Theorem 2.17 of \cite{Joyal-Barcelona}.
\end{proof} 

The next lemma has a more general formulation, but the following version will be 
sufficient for our purposes.  

\begin{lemma} 
\label{lem:otimes and hla}
If $i\colon A\to B$ is a monomorphism in $\ssSet$ and $j\colon K\to L$ is left anodyne then the 
induced map 
\[
A\otimes L\cup B\otimes K\to B\otimes L 
\]
is horizontally left anodyne.  
\end{lemma} 

\begin{proof} 
Since both the domain and codomain are cocontinuous functors of $i$ and $j$, 
it suffices to prove the statement in the special case that $i$ is the inclusion $\partial\Delta[m,n]\subset 
\Delta[n]$ and $j$ is the horn inclusion $\Lambda^k[p]\subset \Delta[p]$ for $0\leq k<p$.  In this 
case the statement follows as in the proof of the previous lemma using Corollary 2.1.2.7 of \cite{HTT} 
or Theorem 2.17 of \cite{Joyal-Barcelona}.   
\end{proof} 

Our next goal is to prove the following theorem.   

\begin{theorem} 
\label{thm:diag of Rezk lf}
Let $p\colon X\to Y$ be a strong Reedy left fibration.  Then  
the diagonal $dp\colon dX\to dY$ is a left fibration in $\sSet$.  
\end{theorem} 

Before we give the proof of this theorem let us borrow and abuse some 
convenient notation from Joyal and Tierney.  If $i\colon A\to B$ 
and $p\colon X\to Y$ are maps in $\ssSet$ let us write 
\[
\langle i,p\rangle \colon \map(B,X)\to 
\map(A,X)\times_{\map(A,Y)} \map(B,Y) 
\]
for the canonical map in $\sSet$ induced from the commutative diagram 
\[
\begin{tikzcd} 
\map(B,X) \arrow[r] \arrow[d] & \map(A,X) \arrow[d] \\ 
\map(B,Y) \arrow[r] & \map(A,Y) 
\end{tikzcd} 
\]
in $\ssSet$.  

\begin{proof} 
Let $\cA$ be the class of all monomorphisms $i$ in $\sSet$ such that 
$\langle \delta_!(i),p\rangle$ is a trivial Kan fibration, where $\delta_!\colon \sSet\to \ssSet$ 
denotes the left adjoint to the diagonal functor $d\colon \ssSet\to \sSet$.  Since a 
trivial Kan fibration is surjective on vertices, to prove the theorem it is sufficient to prove 
that every left anodyne map in $\sSet$ is contained in $\cA$.  
Therefore, by Proposition~\ref{prop:contains left anodynes}, it is sufficient to 
prove that $\cA$ is saturated, $\cA$ satisfies the right cancellation property, 
and that the initial vertex maps $\delta_n\colon \Delta[0]\to \Delta[n]$ are contained in 
$\cA$ for every $n\geq 0$.  

We show that $\cA$ is saturated.  The class of trivial Kan fibrations 
is closed under arbitrary products, pullbacks and retracts.  We show that 
it is closed under sequential composition, which amounts to showing that given a sequence 
$\cdots \to X_n\to X_{n-1}\to \cdots \to X_0$ of trivial Kan fibrations, 
the canonical map $X \to X_0$ 
is a trivial Kan fibration, where $X = \varprojlim X_n$.  Recall that the inverse limit functor 
$\varprojlim \colon \sSet^{\NN}\to \sSet$ is right Quillen for the injective 
model structure on $\sSet^{\NN}$, and that a map $X\to Y$ in $\sSet^{\NN}$ is a 
trivial fibration in the injective model structure if and only if $X_0\to Y_0$ is a trivial Kan fibration, and 
$X_{n+1}\to Y_{n+1}\times_{Y_n} X_n$ is a trivial Kan fibration for all $n\geq 0$ 
(see for example \cite{GJ} VI Proposition 1.3).  These conditions are trivially 
satisfied in our case, hence the result.

We show that $\cA$ satisfies the right cancellation property.  Suppose $u\colon A\to B$ 
and $v\colon B\to C$ are monomorphisms in $\sSet$ such that $vu,u\in \cA$.  
Since $\delta_!\colon \sSet\to \ssSet$ sends monomorphisms in $\sSet$ to 
monomorphisms in $\ssSet$ (Proposition~\ref{prop:delta!(mono)}) 
and $p$ is a horizontal Reedy left fibration it follows from Lemma~\ref{lem:otimes and hla} that 
$\langle \delta_!v,p\rangle$ is a left fibration.  Observe that $\langle \delta_!(vu),p\rangle 
= \langle \delta_!u,p\rangle \langle \delta_!v,p\rangle$ and that $\langle \delta_!(vu),p\rangle$, 
$\langle \delta_!u,p\rangle$ are trivial Kan fibrations by hypothesis.  It follows easily, using the fact that 
a left fibration is a trivial Kan fibration if and only if its fibers are weakly contractible, that 
$\langle \delta_!v,p\rangle$ is a trivial Kan fibration. 

%

Let $n\geq 0$; we show that $\delta_n\colon \Delta[0]\to \Delta[n]$ belongs to 
$\cA$.  This amounts to showing that the induced map $\delta_!(\delta_n)\colon \Delta[0,0]
\to \Delta[n,n]$ is such that $\langle \delta_!(\delta_n),p\rangle$ is a trivial 
Kan fibration.  But $\delta_!(\delta_n)$ factors as 
$\Delta[0,0]\xrightarrow{i} \Delta[n,0]\xrightarrow{j} \Delta[n,n]$.  Therefore it suffices 
to show that $\langle i,p\rangle$ and $\langle j,p\rangle$ are trivial 
Kan fibrations.  By Lemma~\ref{lem:slice simp adjunction}, the map $\langle j,p\rangle$ 
is isomorphic to the map 
\[
\map(\Delta[n],X_{n*}) \to 
\map(\Delta[n],Y_{n*})\times_{\map(\Delta[0],Y_{n*})}\map(\Delta[0],X_{n*}) 
\]
%
Since $\delta_n\colon \Delta[0]\to \Delta[n]$ 
is left anodyne and $X_{n*}\to Y_{n*}$ is a 
left fibration (Lemma~\ref{lem:left slice of horiz lf is lf}) it follows that $\langle j,p\rangle$ is a trivial Kan fibration.  
Also, 
by hypothesis, the induced maps $X_{n*}\to X_{0*}\times_{Y_{0*}}Y_{n*}$ are 
trivial Kan fibrations for all $n\geq 0$ and therefore Lemma~\ref{lem:slice simp adjunction} 
implies that $\langle i,p\rangle$ 
is a trivial Kan fibration which completes the proof.    
\end{proof} 

Next we show that the diagonal functor $d$ sends column-wise cofinal maps to 
cofinal maps of simplicial sets.  

\begin{proposition} 
\label{prop:d of col wise initial map}
Let $f\colon X\to Y$ be a map of bisimplicial sets which is column-wise right cofinal.  
Then the diagonal $df\colon dX\to dY$ is a right cofinal map of simplicial sets.  
\end{proposition} 

\begin{proof} 
Factor $f$ as $X\to Z\to Y$, where $X\to Z$ is horizontally left anodyne and $Z\to Y$ is a 
horizontal Reedy left fibration.  We will prove below that $Z\to Y$ is in fact a trivial Reedy fibration.  
The diagonal $dX\to dZ$ is left anodyne by Lemma~\ref{lem:diag of la map} and the diagonal $dZ\to dY$ is a 
trivial Kan fibration by Corollary~\ref{corr:d sends triv fibns to triv fibns}.  
Therefore the composite map $dX\to dY$ is right cofinal 
(Definition~\ref{def:right cof map}).  

We prove that $Z\to Y$ is a trivial Reedy fibration.  Therefore we need to prove that the map 
\[
p\colon Z_{m*}\to \partial\Delta[m]\backslash Z\times_{\partial\Delta[m]\backslash Y}Y_{m*} 
\]
is a trivial Kan fibration for all $m\geq 0$.  By hypothesis $p$ is a left fibration.  Therefore it suffices 
to prove that its fibers are contractible.  Observe that $Z_{m*}\to Y_{m*}$ is right cofinal
for all $m\geq 0$, since the class of right cofinal maps satisfies the right cancellation property.  
Therefore $Z_{m*}\to Y_{m*}$ is a trivial Kan fibration, since a left fibration 
which is also right cofinal is a trivial Kan fibration (Lemma~\ref{lem:left fibn + right cofinal => trivial Kan fibn}).  

It suffices to prove that the map $\partial\Delta[m]\backslash Z\to \partial\Delta[m]\backslash Y$ 
is a trivial Kan fibration, since a standard argument shows that the fibers of $p$ are contractible.  
More generally we prove that $A\backslash Z\to A\backslash Y$ is a trivial Kan fibration 
for any simplicial set $A$.  

To prove this statement we first use the fact that the functors $(-)\backslash Z\colon \sSet^{\op}\to \sSet$ and 
$(-)\backslash Y\colon \sSet^{\op}\to \sSet$ are continuous, to prove by induction on $n$ that 
$\sk_nA\backslash Z\to \sk_nA\backslash Y$ is a trivial Kan fibration for any simplicial 
set $A$, where $\sk_nA$ denotes the $n$-skeleton.  The initial case $n=0$ is easy, and the 
inductive step follows from a combination of Lemma~\ref{lem:left slice of horiz lf is lf}, 
the fact that $\partial\Delta[n] = \sk_{n-1}\Delta[n]$, and the following trivial observation: suppose that 
\[
\begin{tikzcd} 
W\ar[r] \arrow[d,"p"'] & Y \arrow[d,"q"] \\ 
Z\arrow[r] & X 
\end{tikzcd} 
\]
is a commutative diagram of simplicial sets in which $p$ and $q$ are trivial Kan fibrations, then the 
canonical map $W\to Z\times_X Y$ is a  trivial Kan fibration if it is a left fibration (the hypotheses on $p$ 
and $q$ imply that it is a fiberwise weak homotopy equivalence).  

To complete the proof it suffices to show that
\[
\sk_{n+1}A\backslash Z\to sk_{n+1}A\backslash Y\times_{\sk_nA\backslash Y} 
\sk_nA \backslash Z 
\]
is a trivial Kan fibration for every $n\geq 0$ (compare 
with the proof of Theorem~\ref{thm:diag of Rezk lf} above).  
But this follows from Lemma~\ref{lem:left slice of horiz lf is lf} 
and the observation above.     
\end{proof}

\section{Covariant model structures and simplicial presheaves} 
\label{sec:main}
\subsection{The projective model structure on $[(\Delta/B)^{\op},\sSet]$} 
\label{subsec:submain1}
Let $B$ be a simplicial set.  The Yoneda embedding induces an embedding 
$y/B\colon \Delta/B\to \sSet/B$
and hence we obtain 
an adjoint pair 
\[
(y/B)_!\colon [(\Delta/B)^{\op},\Sets] \rightleftarrows \sSet/B\colon (y/B)_* 
\]
in which the left adjoint functor $(y/B)_!$ is the left Kan extension 
along $y/B$.  It is well known that the pair $((y/B)_!,(y/B)_*)$ 
is an adjoint equivalence.  Clearly there is an induced adjoint equivalence 
\[
(y/B)_!\colon [(\Delta/B)^{\op},\sSet]\rightleftarrows \ssSet/B\Box 1 \colon (y/B)_*
\]
between the corresponding categories of simplicial objects.  If $X\in \ssSet/B\Box 1$ then 
$(y/B)_*(X)$ is the simplicial presheaf on $\Delta/B$ whose simplicial set of sections over 
$\sigma\colon \Delta[m]\to B$ is isomorphic to 
\[
X_{m*}\times_{B_m}\set{\sigma}.  
\]  
It follows easily that $(y/B)_*$ 
sends horizontal projective fibrations (respectively trivial fibrations) in $\ssSet/B\Box 1$ 
to projective fibrations (respectively trivial fibrations) in $[(\Delta/B)^{\op},\sSet]$.  Therefore we have the following 
result.  

\begin{theorem} 
\label{thm:equiv of spshvs with ssSets}
Let $B$ be a simplicial set.  Then the adjunction 
\[
(y/B)_!\colon [(\Delta/B)^{\op},\sSet]\rightleftarrows \ssSet/B\Box 1\colon (y/B)_* 
\]
is a Quillen equivalence for the projective model structure on $[(\Delta/B)^{\op},\sSet]$ 
and the horizontal projective model structure on $\ssSet/B\Box 1$.  
\end{theorem} 

The adjunction $((y/B)_!,(y/B)_*)$ is simplicial for the horizontal simplicial enrichment on 
$\ssSet/B\Box 1$ and the canonical simplicial enrichment on $[(\Delta/B)^{\op},\sSet]$ for which 
the tensor $F\otimes K$ for a simplicial presheaf $F$ and a simplicial set $K$ is the simplicial 
presheaf $F\otimes K$ defined by 
\[
(F\otimes K)(\sigma) = F(\sigma)\times K.  
\]
 
The category $[(\Delta/B)^{\op},\sSet]$, equipped 
with the projective model structure, is what is called in \cite{Dugger} the {\em universal 
homotopy theory} built from $(\Delta/B)^{\op}$.  It follows from Proposition 2.3 of \cite{Dugger} 
that the embedding $y/B\colon \Delta/B\to \sSet/B$ induces a Quillen adjunction 
\[
\Re_B\colon [(\Delta/B)^{\op},\sSet]\rightleftarrows \sSet/B\colon \Sing_B 
\]
for the covariant model structure on $\sSet/B$.  When $B$ is understood we will simply write 
$\Re$ and $\Sing$.  The adjoint pair $(\Re,\Sing)$ 
factors through the adjoint pair $(d,d_*)$ of Proposition~\ref{prop:comparison} 
in the sense that there are natural isomorphisms $\Sing  \simeq (y/B)_*d_*$ and 
$\Re \simeq d(y/B)_!$.  
Hence in the diagram 
\[
\xymatrix{
[(\Delta/B)^{\op},\sSet] \ar@<0.5ex>[r]^-{(y/B)_!} & \ssSet/B\Box 1 \ar@<0.5ex>[l]^-{(y/B)_*} 
\ar@<0.5ex>[r]^-{d} & \sSet/B \ar@<0.5ex>[l]^-{d_*} } 
\]
the composite horizontal right pointing arrow is naturally isomorphic to 
$\Re\colon [(\Delta/B)^{\op},\sSet]\to \sSet/B$ and the composite 
horizontal left pointing arrow is naturally isomorphic to 
$\Sing\colon \sSet/B\to [(\Delta/B)^{\op},\sSet]$.

\subsection{The localized projective model structure on $[(\Delta/B)^{\op},\sSet]$}
\label{subsec:loc proj model structure}
We denote by $W_B$ the 
wide subcategory of the simplex category $\Delta/B$ of $B$ whose maps are the initial 
vertex maps in $\Delta$.  When $B$ is clear from the context and no confusion is likely we will 
simply write $W$ instead of $W_B$.  

It is clear that each map 
in $W$ is sent to a left anodyne map in $\sSet/B$ under the embedding 
$\Delta/B\to \sSet/B$.  It follows that the Quillen adjunction $(\Re,\Sing)$ descends to a 
Quillen adjunction  
\[
\Re\colon L_W[(\Delta/B)^{\op},\sSet]\rightleftarrows \sSet/B\colon \Sing,
\] 
where $L_W[(\Delta/B)^{\op},\sSet]$ denotes the left Bousfield localization of the 
projective model structure on $[(\Delta/B)^{\op},\sSet]$ along the image of $W$ under the 
Yoneda embedding $y\colon \Delta/B\to [(\Delta/B)^{\op},\sSet]$.    
In this subsection we will prove that in fact this is a Quillen equivalence.   

Before we proceed to the proof of this fact, let us spell out what it means for a 
simplicial presheaf $F\colon (\Delta/B)^{\op}\to \sSet$ to be $W$-local.  By definition, $F$ is 
$W$-local if and only if $F$ is projectively fibrant and the induced map  
\[
u^*\colon \map(\sigma',F)\to \map(\sigma,F) 
\]
is a weak homotopy equivalence for all maps $u\colon \sigma\to \sigma'$ in $W$, where 
$\map(-,-)$ denotes the simplicial enrichment in $[(\Delta/B)^{\op},\sSet]$.  Since the canonical 
initial vertex map $\sigma(0) = \sigma'(0) \to \sigma'$ factors through $u$ via the initial vertex map 
$\sigma(0)\to \sigma$, we see that a projectively fibrant simplicial presheaf $F$ is 
$W$-local if and only if the map 
\[
\map(\sigma,F)\to \map(\sigma(0),F) 
\]
induced by the initial vertex map $\sigma(0)\to \sigma$ is a weak homotopy equivalence for 
all objects $\sigma\in \Delta/B$ . Alternatively, if we 
identify $F$ with an object $F\in \ssSet/B\Box 1$ by means of the equivalence 
of Theorem~\ref{thm:equiv of spshvs with ssSets}, i.e.\ if we identify 
$F\in [(\Delta/B)^{\op},\sSet]$ with $(y/B)_!(F)\in \ssSet/B\Box 1$, then $F$ is $W$-local 
if and only if $F$ is projectively fibrant in $\ssSet/B\Box 1$, and the maps 
\[
F_{n*}\to F_{0*}\times_{B_0} B_n 
\]
induced by $0\colon \Delta[0]\to \Delta[n]$ are weak homotopy equivalences 
for all $n\geq 0$.  We will now prove Theorem~\ref{thm:main} from the introduction.  
Recall the statement of this theorem.    

\theoremA*

\begin{proof}
We begin by proving that if $X\in \bL(B)$, then the map 
$\Re\,  Q\, \Sing(X)\to X$ is a covariant equivalence in $\sSet/B$, where 
$Q\, \Sing(X)\to \Sing(X)$ is a cofibrant replacement in $[(\Delta/B)^{\op},\sSet]$.  
Equivalently, by Theorem~\ref{thm:equiv of spshvs with ssSets}, we may prove that the map $dQd_*X\to X$ is a 
covariant equivalence in $\sSet/B$, where $Qd_*X\to d_*X$ is a cofibrant replacement 
in $\ssSet/B\Box 1$.  But $dQd_*X\to dd_*X$ 
is a covariant equivalence (Proposition~\ref{prop:comparison}), and 
hence it suffices to prove that $dd_*X\to X$ is a covariant equivalence.  
Observe that there is a 
factorization 
\[
X\to dd_*X\to X 
\]
of the identity map of $X$, where the map $X\to dd_*X$ is the 
diagonal of the canonical map of bisimplicial sets $X\to d_*X$, which in degree 
$n$ is the map $X\to (d_*X)_{*n}$ induced by the unique map 
$[n]\to [0]$.  Therefore it is 
sufficient to prove that $X\to dd_*X$ is a covariant equivalence.    
By Lemma~\ref{lem:diags of row-wise cov eqs are cov eqs} it 
is sufficient to show that the maps $X\to (d_*X)_{*n}$ are 
covariant equivalences in $\sSet/B$ for all $n$.  But for every $n\geq 0$ the  
map $X\to (d_*X)_{*n}$ is left inverse to the map 
\[
X^{\Delta[n]}\to X 
\]
in $\bL(B)$ induced by the map $0\colon [0]\to [n]$.  This last map is a trivial 
fibration since $0\colon \Delta[0]\to \Delta[n]$ is left anodyne and $X\in \bL(B)$. Hence $X\to (d_*X)_{*n}$ 
is a covariant equivalence.

To complete the proof we show that for any projectively cofibrant 
$F\in [(\Delta/B)^{\op},\sSet]$, and for a suitable choice of fibrant 
replacement $\Re(F)\to R\, \Re(F)$, the map 
\[
F\to \Sing(R\, \Re(F)) 
\]
in $[(\Delta/B)^{\op},\sSet]$ is a $W$-local equivalence.  
Regard $F$ as an object of $\ssSet/B\Box 1$.  
Without loss of generality we may suppose that $F$ is $W$-local.  
We may factor the structure map $F\to B\Box 1$ as $F\to F'\to B\Box 1$, 
where $F\to F'$ is a horizontal left anodyne map and $F'\to B\Box 1$ 
is a horizontal Reedy left fibration.  Since $F$ is $W$-local the maps 
\[
F_{n*}\to B_n\times_{B_0} F_{0*} 
\]
induced by $0\colon \Delta[0]\to \Delta[n]$ are weak homotopy equivalences 
for all $n\geq 0$.  Therefore, since $F\to F'$ is column-wise left anodyne (Lemma~\ref{lem:horiz la implies col-wise la}), 
the vertical maps in the commutative diagram 
\[
\begin{tikzcd} 
F_{n*} \arrow[r] \arrow[d] & B_n\times_{B_0} F_{0*} \arrow[d] \\ 
F'_{n*} \arrow[r] & B_n\times_{B_0} F'_{0*} 
\end{tikzcd} 
\]
are weak homotopy equivalences, and hence 
\[
F'_{n*}\to B_n\times_{B_0} F'_{0*} 
\]
is a weak homotopy equivalence for every $n\geq 0$.  Therefore we 
may apply Lemma~\ref{lem:recognize str lf} to deduce that 
$F'\to B\Box 1$ is a strong horizontal Reedy left fibration.  

We have a commutative diagram 
of the form 
\[
\begin{tikzcd}
F\arrow[d] \arrow[r] & F' \arrow[d] \\ 
d_*dF \arrow[r] & d_*dF' 
\end{tikzcd} 
\]
in $\ssSet/B$.  We will prove the following statements are true: 
\begin{enumerate} 
\item $dF\to dF'$ is a covariant equivalence in $\sSet/B$; 
\item $dF'\in \bL(B)$ (therefore $dF'$ is a fibrant replacement of $dF$ in $\sSet/B$); 
\item $F'\to d_*dF'$ is a column-wise weak homotopy equivalence.  
\end{enumerate} 
These statements suffice to prove that $F\to \Sing(R\, \Re F)$ is a $W$-local equivalence.  
Statement (1) follows from Lemma~\ref{lem:diag of la map}.  
Statement (2) follows from Theorem~\ref{thm:diag of Rezk lf}.  We need 
to prove Statement (3).  Note that $F'$ is $W$-local, and 
also $d_*dF'$ is $W$-local (Lemma~\ref{lem:d_*X is W local}).  Therefore, it is 
sufficient to prove that $F'_{0*}\to (d_*dF')_{0*}$ is a weak homotopy equivalence.  
Recall from the proof of Lemma~\ref{lem:d_*X is W local} that $(d_*dF')_{0*}$ 
forms part of a pullback diagram 
\[
\begin{tikzcd}
(d_*dF')_{0*} \arrow[r] \arrow[d] & dF' \arrow[d] \\ 
B_0 \arrow[r] & B 
\end{tikzcd} 
\]
and so $(d_*dF')_{0*}$ is the diagonal of the bisimplicial set whose 
$m$-th row is 
\[
B_0\times_{B_m} F'_{m*}
\]
and the map $F'_{0*}\to (d_*dF)_{0*}$ is the diagonal of the map of 
bisimplicial sets which on $m$-th rows is the 
map 
\[
F'_{0*}\to B_0\times_{B_m} F'_{m*} 
\]
induced by the canonical map $[m]\to [0]$.  Clearly it is sufficient to prove 
that each of these maps on rows is a weak homotopy equivalence.  But the 
composite 
\[
F'_{0*}\to B_0\times_{B_m} F'_{m*} \to F'_{0*} 
\]
is the identity, and the second map is a weak homotopy equivalence since 
$F'$ is $W$-local, hence the result.  
\end{proof} 

\begin{remark}
\label{rem:cov equiv}
Notice that one outcome of this proof is that the counit map $\Re\, \Sing(X)\to X$ 
is a covariant equivalence for any $X\in \sSet/B$, which is a stronger statement than what 
one would usually expect.  
\end{remark}
  
A dual version of the {\em mapping simplex} $M(\phi)$ of a map $\phi\colon [n]\to \sSet$ (see 
Section 3.2.2 of \cite{HTT}) can be understood in terms of the functor $\Re$ in the case 
where $B = \Delta[n]$.  If $q_{\Delta[n]}\colon (\Delta/[n])^{\op}\to [n]$ denotes the 
map from introduction defined by $q_{\Delta[n]}(u) = u(0)$ for $u\colon [m]\to [n]$, then 
the (dual) of the mapping simplex $M(\phi)$ of a map $\phi\colon [n]\to \sSet$ 
is $M(\phi) = \Re(q_{\Delta[n]}^*\phi)$.  This observation, together with Theorem~\ref{thm:main}, 
can be used to prove a version of Proposition 3.2.2.7 of \cite{HTT} for left fibrations. 

\begin{remark}
\label{rem:naturality of main thm}
If $f\colon A\to B$ is a map of simplicial sets then $f$ induces a Quillen adjunction 
$f_!\colon \sSet/A\rightleftarrows \sSet/B\colon f^*$ between the respective covariant 
model structures and (since $f_!\colon \Delta/A\to \Delta/B$ maps the initial vertex maps 
in $\Delta/A$ to the initial vertex maps in $\Delta/B$) a Quillen adjunction 
$f_!\colon L_{W_A}[(\Delta/A)^{\op},\sSet]\rightleftarrows L_{W_B}[(\Delta/B)^{\op},\sSet]\colon f^*$ 
between the localized projective model structures.  Moreover, the diagram 
\[
\begin{tikzcd} 
L_{W_A}[(\Delta/A)^{\op},\sSet] \arrow[d,"f_!"'] \arrow[r] & \sSet/A \arrow[d,"f_!"] \\ 
L_{W_B}[(\Delta/B)^{\op},\sSet] \arrow[r] & \sSet/B 
\end{tikzcd} 
\]
of left Quillen functors commutes up to natural isomorphism (the horizontal maps 
are the left Quillen functors from Theorem~\ref{thm:main}).  
\end{remark}

\subsection{Another Quillen equivalence} 
\label{sec:presheaf straightening}  
By composing the embedding $y/B$ with the forgetful functor we obtain a simplicial diagram $y/B\colon \Delta/B\to \sSet$ 
on $\Delta/B$.  As such we may form its {\em simplicial replacement} $s(y/B)$; this is a bisimplicial 
set whose $n$-th row is 
\[
s(y/B)_n = \bigsqcup_{\sigma_0\to \cdots \to \sigma_n} y/B(\sigma_0) 
\]
where the coproduct is taken over the set of $n$-simplices in the nerve $N(\Delta/B)$.  
The various maps $\sigma_0\colon \Delta[m_0]\to B$ define a row augmentation 
$s(y/B)\to B\Box 1$.  If $X\in \sSet/B$ then we can form the bisimplicial set 
\[
s(X):= s(y/B)\times_{B\Box 1}X\Box 1,
\]
which is again row-augmented over $B$.  The construction $s(X)$ defines a functor 
\[
s\colon \sSet/B\to \ssSet/B\Box 1;
\]
since $s(X)$ is constructed as a fiber product it follows that the 
functor $s$ is cocontinuous.  The bisimplicial set $s(X)$ is canonically isomorphic to the bisimplicial 
set whose $n$-th row is 
\[
\bigsqcup_{y/B(\sigma_0)\to \cdots \to y/B(\sigma_n) \to X} (y/B)_!(y(\sigma_0)) 
\]
where $(y/B)_!$ is the functor from Section~\ref{subsec:submain1}.
Let 
\[
s_!\colon \sSet/B\to [(\Delta/B)^{\op},\sSet] 
\]
denote the functor which sends $X\in \sSet/B$ to the simplicial presheaf $s_!X$ on 
$\Delta/B$ whose presheaf of $n$-simplices is 
\[
s_!(X)_n  = \bigsqcup_{y(\sigma_0)\to \cdots \to y(\sigma_n) \to (y/B)_*X} y(\sigma_0). 
\]
Then $s_!$ forms part of an adjoint pair $(s_!,s^!)$ and moreover the composite functor 
$(y/B)_!s_!$ is naturally isomorphic to the functor $s\colon \sSet/B\to \ssSet/B\Box 1$ described above.

Observe that for any simplicial set $X\in \sSet/B$ there is a canonical map 
\[
ds(X)\to X, 
\]
natural in $X$.  This map is obtained as the diagonal of the canonical map 
$s(X)\to X$ of bisimplicial sets (here $X$ is regarded as a horizontally 
constant bisimplicial set), which on $n$-th rows is the map 
\[
\bigsqcup_{\sigma_0\to \cdots \to \sigma_n} y/B(\sigma_0) \to X 
\]
induced by the map $\sigma_0\colon \Delta[m_0]\to X$, 
where the coproduct is indexed over the set of $n$-simplices of the nerve 
of the category $\Delta/X$ of simplices of $X$.  

We have the following result.  

\begin{proposition} 
\label{prop:dtildeQX to X initial}
Let $B$ be a simplicial set.  Then the canonical map $ds(X)\to X$ is a 
covariant equivalence in $\sSet/B$ for any $X\in \sSet/B$.  
\end{proposition} 

\begin{proof} 
We prove that the map on $m$-th columns 
\[
s(X)_{m*} \to X_m 
\]
is right cofinal for every $m\geq 0$.  The result then follows by applying 
Proposition~\ref{prop:d of col wise initial map} and the fact that a right cofinal map 
in $\sSet/B$ is a covariant equivalence (Lemma~\ref{lem:cov equiv right cof}).  

Observe that we may regard $X_m$ as a discrete category and that 
$s(X)_{m*}$ is the nerve of a category whose objects 
are pairs $(\sigma,x)$ where $\sigma\colon \Delta[n_\sigma]\to X$ is a simplex of $X$  
 and $x\colon \Delta[m]\to \Delta[n_\sigma]$ 
is an $m$-simplex.  A morphism $(\sigma,x)\to (\tau,y)$ between such pairs consists of a 
map $u\colon \Delta[n_\sigma]\to \Delta[n_\tau]$ in the simplex category $\Delta/B$ such that 
$ux=y$.  

We use Theorem~\ref{thm:A}.  It suffices to show that for every $m$-simplex 
$x\in X_m$, the fiber $\pi^{-1}(x)$ of the map $\pi\colon r(X)_{m*}\to X_m$ over $x$ is weakly contractible.  
But $(x\colon \Delta[m]\to X,\mathrm{id})$ is an initial object of $\pi^{-1}(x)$, where $\mathrm{id}\colon \Delta[m]\to 
\Delta[m]$ denotes the unique non-degenerate $m$-simplex.   The result follows.  
\end{proof} 
  
Our next aim is to show that the  functor $s_!\colon \sSet/B\to [(\Delta/B)^{\op},\sSet]$ is a left Quillen functor 
for the covariant model structure on $\sSet/B$ and the projective model structure on $[(\Delta/B)^{\op},\sSet]$.  
In fact, if $X\subset Y$ is a monomorphism in $\sSet/B$, 
then $s_!X\to s_!Y$ is a {\em degeneracy free} morphism (Definition VII 1.10 of \cite{GJ}) with a decomposition 
$s_!Y_n = s_!X_n\sqcup Z_n$, where  
\[
Z_n = \bigsqcup_{\sigma_0 \to \cdots \to \sigma_n} \Delta[n_0] \subset s_!Y_n,  
\]
and where for each summand at least one $\sigma_i\colon \Delta[n_i]\to Y$ is a simplex of $Y$ which does not belong to 
$X$.  It follows that the functor $s_!\colon \sSet/B\to L_W[(\Delta/B)^{\op},\sSet]$ sends covariant 
cofibrations to cofibrations in the localized projective model structure on $[(\Delta/B)^{\op},\sSet]$ 
(see the discussion in Example VII 1.15 from \cite{GJ}).    
We will now prove 

\theoremB*

\begin{proof} 
We prove first that $(s_!,s^!)$ is a Quillen adjunction.  
By the remarks above, it suffices to prove that if $X\to Y$ is a covariant equivalence in $\sSet/B$ then 
$s(X)\to s(Y)$ is a $W$-local equivalence in $[(\Delta/B)^{\op},\sSet]$.  Equivalently, we need 
to show that $ds(X)\to ds(Y)$ is a covariant equivalence, since $d\colon L_W[(\Delta/B)^{\op},\sSet]\to \sSet/B$ 
reflects weak equivalences between cofibrant objects by Theorem~\ref{thm:main}.  The result then follows by 
Proposition~\ref{prop:dtildeQX to X initial}.    

Now we prove that $(s_!,s^!)$ is a Quillen equivalence.  
We show first that $s_!$ reflects weak equivalences between 
cofibrant objects.  Equivalently, by 
Theorem~\ref{thm:equiv of spshvs with ssSets}, we may prove that $s\colon \sSet/B\to \ssSet/B\Box 1$ 
reflects weak equivalences between cofibrant objects.  
Suppose that $X\to Y$ is a map in $\sSet/B$ such that 
$s(X)\to s(Y)$ is a $W$-local equivalence in $[(\Delta/B)^{\op},\sSet]$.  Then 
$ds(X)\to ds(Y)$ is a covariant equivalence, from which it follows that 
$X\to Y$ is a covariant equivalence by Proposition~\ref{prop:dtildeQX to X initial}.  

Let $F\in [(\Delta/B)^{\op},\sSet]$ and suppose that $F$ is $W$-local.  We 
will prove that $s_!s^!F\to F$ is a $W$-local equivalence.  It suffices to 
prove that $s_!s^!\Sing(X)\to \Sing(X)$ is a $W$-local equivalence for any 
$X\in \bL(B)$.  Therefore, by Remark~\ref{rem:cov equiv}, it suffices to prove that $\Re\, s_!s^!\Sing(X) \to  \Re\, \Sing(X)$ 
is a covariant equivalence.  By Proposition~\ref{prop:dtildeQX to X initial} it suffices to prove that 
the map $X\to s^!\Sing(X)$, conjugate to 
$\Re\, s_!X\to X$, is a covariant equivalence for any $X\in \bL(B)$.  
To do this it is sufficient to prove that for every $Y\in \sSet/B$ the induced map 
\[
\ho(\sSet/B)(Y,X)\to \ho(\sSet/B)(Y,s^!\Sing(X)) 
\]
is an isomorphism.  But this map is isomorphic to the map 
\[
\ho(\sSet/B)(Y,X)\to \ho(\sSet/B)(\Re\, s_!Y,X) 
\]
induced by $\Re\, s_!Y\to Y$.  The result then follows from 
Proposition~\ref{prop:dtildeQX to X initial}.    
\end{proof}

\section{Localization of simplicial categories and quasi-categories} 
\label{sec:localization}

\subsection{Simplicial localization} 
\label{subsec:simp loc}
Write $\SCat$ for the category of simplicial categories and simplicial functors 
between them; if $O$ is a set then we will write $\SCat(O)$ for the subcategory 
of $\SCat$ on the simplicial categories $C$ such that $\Ob(C) = O$.  The morphisms in $\SCat(O)$ are the functors 
which are the identity on objects.  
Recall (see Proposition 7.2 of \cite{DK1}) that the category $\SCat(O)$ has 
the structure of a simplicial model category for which the weak equivalences 
are the maps $A\to B$ in $\SCat(O)$ such that $A(x,y)\to B(x,y)$ is a weak 
homotopy equivalence for all $x,y\in O$.  More generally in the {\em Bergner 
model structure}  \cite{Bergner} on $\SCat$  the weak equivalences are the {\em DK-equivalences}, 
i.e.\ the maps $f\colon A\to B$ in $\SCat$ such that (i) $A(x,y)\to B(f(x),f(y))$ is a weak homotopy 
equivalence for all objects $x,y$ of $A$ and (ii) $\pi_0 f\colon \pi_0A\to \pi_0B$ is an 
equivalence of categories (see \cite{Bergner} for more details).

We recall some of the details of the simplicial localization functor described 
in \cite{DK1}.  

\begin{definition} 
If $C\in \SCat(O)$ and $W\subset C$ is a subcategory in 
$\SCat(O)$ (so that $\Ob(W) = O$) then the {\em simplicial localization} $L^S(C,W)$  is defined as follows.  
In each degree $n\geq 0$ form the free simplicial $O$-categories $F_*W_n$ 
and $F_*C_n$ ((2.5) of \cite{DK1}).  Thus we have simplicial objects $F_*W$ and $F_*C$ in $\SCat(O)$.  
Then $L^S(C,W)$ is defined to be 
\[
L^S(C,W) = d(F_*C[F_*W^{-1}]), 
\]
where $d\colon s\SCat(O)\to \SCat(O)$ denotes the diagonal functor.  
\end{definition}

Here if $A\subset B$ is a simplicial subcategory of $B$ in $\SCat(O)$ then we write 
$B[A^{-1}]$ for the simplicial category in $\SCat(O)$ defined by the pushout diagram 
\[
\begin{tikzcd} 
A \arrow[r] \arrow[d] & B \arrow[d] \\ 
A[A^{-1}] \arrow[r] & B[A^{-1}] 
\end{tikzcd} 
\]
in $\SCat(O)$, where $A\to A[A^{-1}]$ is the unit of the adjunction

\[
(-)[(-)^{-1}]\colon \SCat(O)\rightleftarrows \SGpd(O)\colon i,
\]
in which 
$i\colon \SGpd(O)\subset \SCat(O)$ denotes the inclusion of the category $\SGpd(O)$ 
of simplicial groupoids into the category $\SCat(O)$.

We shall also make use of the {\em hammock localization} $L^H(C,W)$ defined in \cite{DK5}; we refer 
the reader to this paper for the details of this construction as these details will not play an important role for us.     
We recall that there is a zig-zag of DK-equivalences relating the simplicial localization 
$L^S(C,W)$ and the hammock localization $L^H(C,W)$: 
\[
L^H(C,W) \leftarrow \mathrm{diag}L^H(F_*C,F_*W) \to L^S(C,W) 
\]
(see Proposition 2.2 of \cite{DK5}).    

If $W$ is a collection of arrows in a simplicial category $C$, then a different version of the simplicial 
localization of $C$ along $W$ is defined by Lurie in Section A.3.5 of \cite{HTT} 
as follows.  

\begin{definition}[Lurie]
\label{def:luries simp loc}
If $W$ is a collection of arrows in a simplicial category $C$,  then the simplicial localization 
$L(C,W)$ of $C$ along $W$ may be 
taken to be the 
pushout 
\[
\begin{tikzcd}
\bigsqcup_{w\in W} I \arrow[d] \arrow[r] & C \arrow[d] \\ 
\bigsqcup_{w\in W} \mathfrak{C}[J] \arrow[r] & L(C,W) 
\end{tikzcd} 
\]
where the map $I\to \mathfrak{C}[J]$ is the cofibration obtained by applying the 
functor $\mathfrak{C}[-]\colon \sSet\to \SCat$ to the canonical inclusion 
of simplicial sets $I\to J$.    
\end{definition}

Recall in the definition above that $J$ denotes the groupoid interval, i.e.\ the nerve of the groupoid with 
two objects and one isomorphism between them.   

The simplicial category $\mathfrak{C}[S]$ is difficult to describe explicitly 
for arbitrary simplicial sets; in the definition above we have used the easily 
proved fact that the canonical map 
$\mathfrak{C}[I]\to I$ is an isomorphism, where $I$ is regarded as a simplicial category 
whose mapping spaces are discrete simplicial sets.  We shall also need the 
following result, which describes $\mathfrak{C}[S]$ when $S$ is the nerve 
of a small category.  

\begin{proposition}[\cite{Cordier,Riehl}] 
\label{prop:Riehl}
Let $A$ be a small category.  Then there is an isomorphism $\mathfrak{C}[NA] = F_*(A)$.  
In particular the canonical map $\pi_0\colon \mathfrak{C}[NA]\to A$ is a DK-equivalence.  
\end{proposition}        

We refer to \cite{Riehl} for a particularly simple proof of this fact using the technology of `necklaces' 
from \cite{DS}.    

The construction $L(C,W)$ has a number of very useful properties; $L(C,W)$ is given by a homotopy pushout in $\SCat$,  
it can be characterized by a universal property 
(see Proposition A.3.5.5 of \cite{HTT}) and   
$L(C,W)$ is compatible with colimits in the pair $(C,W)$.  




Suppose that $W\subset C$ is a discrete subcategory of $C$ in $\SCat(O)$, so that 
$W(x,y)$ is a discrete simplicial set for all $x,y\in O$.  Then we may consider the 
localization $L(C,W)$ as in Definition~\ref{def:luries simp loc} above (where we have 
abused notation and written $W$ for the set of arrows of the category $W$) 
and compare it with the Dwyer-Kan simplicial localization 
$L^S(C,W)$ as defined in \cite{DK1}.  In the next proposition we show, as one would 
expect, that $L(C,W)$ and $L^S(C,W)$ are weakly equivalent.  

\begin{proposition} 
\label{prop:DK simp loc and Lurie simp loc}
Suppose that $W\subset C$ is a discrete subcategory of $C$ in $\SCat(O)$.  Then 
the simplicial categories $L(C,W)$ and $L^S(C,W)$ are weakly equivalent.  
\end{proposition} 

\begin{proof} 
There is a canonical map $\bigsqcup_{w\in W}
I\to W$; on applying the functor $\mathfrak{C}[-]$ to the corresponding map of 
simplicial sets we obtain a factorization 
\[
\bigsqcup_{w\in W} I\to \mathfrak{C}[W]\to W 
\]
since, as remarked earlier, there is an isomorphism $\mathfrak{C}[I] = I$.  
Factor the resulting map $\mathfrak{C}[W]\to C$ in 
$\SCat(O)$ as a cofibration $\mathfrak{C}[W]\to C'$ 
followed by a trivial DK-fibration $C'\to C$, so that we 
have a commutative diagram 
\[
\begin{tikzcd} 
\bigsqcup_{w\in W} I \arrow[r] \arrow[dr] & \mathfrak{C}[W] \arrow[d] \arrow[r] & C' \arrow[d] \\ 
& W \arrow[r] & C 
\end{tikzcd} 
\]
Note that, since the Bergner model structure on $\SCat$ is left proper (see Proposition A.3.2.4 of 
\cite{HTT}), the canonical map $L(C',W)\to L(C,W)$ is a DK-equivalence.  
We have a commutative diagram 
\[
\begin{tikzcd} 
\bigsqcup_{w\in W} I \arrow[d] \arrow[r] & \mathfrak{C}[W] \arrow[r] \arrow[d] & C' \arrow[d] \\ 
\bigsqcup_{w\in W}\mathfrak{C}[J] \arrow[d] \arrow[r] & L(\mathfrak{C}[W],W) \arrow[r] \arrow[d] & L(C',W) \arrow[d] \\ 
\bigsqcup_{w\in W} J \arrow[r] & \mathfrak{C}[W][\mathfrak{C}[W]^{-1}] \arrow[r] & 
C'[\mathfrak{C}[W]^{-1}] 
\end{tikzcd} 
\]
in which all squares except for the bottom left hand square are pushouts, and 
where the map 
\[
\bigsqcup_{w\in W} J\to \mathfrak{C}[W][\mathfrak{C}[W]^{-1}]
\]
is obtained by applying the groupoid completion functor 
to the map $\sqcup_{w\in W}I\to \mathfrak{C}[W]$.    
The map $L(\mathfrak{C}[W],W)\to L(C',W)$ is a cofibration since it 
is a pushout of the cofibration $\mathfrak{C}[W]\to C'$; therefore to prove that the 
map $L(C',W)\to C'[\mathfrak{C}[W]^{-1}]$ is a 
DK-equivalence it suffices to prove that  $L(\mathfrak{C}[W],W)\to \mathfrak{C}[W] 
[\mathfrak{C}[W]^{-1}]$ is a DK-equivalence.  Since 
$\mathfrak{C}[W][\mathfrak{C}[W]^{-1}]$ is a simplicial groupoid, it follows that   
this last map factors through the groupoid completion of $L(\mathfrak{C}[W],W)$ as 
\[
L(\mathfrak{C}[W],W) \to L(\mathfrak{C}[W],W) 
[L(\mathfrak{C}[W],W)^{-1}] \to 
\mathfrak{C}[W]
[\mathfrak{C}[W]^{-1}]
\]
Observe that $\pi_0L(\mathfrak{C}[W],W) = W[W^{-1}]$ is a 
groupoid; since $L(\mathfrak{C}[W],W)$ is cofibrant it follows that 
the canonical map 
\[
L(\mathfrak{C}[W],W) \to L(\mathfrak{C}[W],W) 
[L(\mathfrak{C}[W],W)^{-1}] 
\]
is a DK-equivalence (Proposition 9.5 of \cite{DK1}).  We consider the second map 
\[
p\colon L(\mathfrak{C}[W],W) 
[L(\mathfrak{C}[W],W)^{-1}] \to 
\mathfrak{C}[W][\mathfrak{C}[W]^{-1}]
\]
in the composite map above.  This map is left inverse to the map 
\[
q\colon \mathfrak{C}[W][\mathfrak{C}[W]^{-1}]\to L(\mathfrak{C}[W],W)[L(\mathfrak{C}[W],W)^{-1}] 
\]
obtained by applying the groupoid completion functor to the map 
$\mathfrak{C}[W]\to L(\mathfrak{C}[W],W)$.  Therefore it suffices to prove 
that $q$ is a DK-equivalence.  The map $q$ forms part of the pushout diagram 
\[
\begin{tikzcd} 
\displaystyle{\bigsqcup_{w\in W}} J \arrow[r] \arrow[d,"i"'] & \mathfrak{C}[W] [\mathfrak{C}[W]^{-1}] \arrow[d,"q"] \\ 
\displaystyle{\bigsqcup_{w\in W}} \mathfrak{C}[J] [\mathfrak{C}[J]^{-1}]  \arrow[r] & 
L(\mathfrak{C}[W],W) [L(\mathfrak{C}[W],W)^{-1}] 
\end{tikzcd} 
\]
in $\SGpd$; we show first that this is a homotopy 
pushout (for the model structure on $\SGpd$ introduced in 
\cite{DK3}).   The map $i$ 
is obtained by applying the groupoid completion functor 
$(-)[(-)^{-1}]$ to the cofibration 
\[
\bigsqcup_{w\in W} I \to \bigsqcup_{w\in W} \mathfrak{C}[J] 
\]
in $\SCat$ and hence is an cofibration in $\SGpd$, since $(-)[(-)^{-1}]\colon 
\SCat\to \SGpd$ is left Quillen.  To complete the proof that $q$ is a DK-equivalence, 
it suffices to show that the left hand vertical map in the diagram above is a DK-equivalence, which is 
straightforward.  

Next, observe that there is a canonical map $L^S(C',\mathfrak{C}[W])\to 
C'[\mathfrak{C}[W]^{-1}]$ and a canonical map 
$L^S(C',\mathfrak{C}[W])\to L^S(C,W)$.  The latter map is a DK-equivalence by 
Corollary 6.3 of \cite{DK1}, since $C'\to C$ is a DK-equivalence which 
restricts to a DK-equivalence $\mathfrak{C}[W]\to W$.  
It follows from Lemma~\ref{lem:comparison of loc 1}, to 
be proved shortly, that 
$L^S(C',\mathfrak{C}[W])\to 
C'[\mathfrak{C}[W]^{-1}]$ is a DK-equivalence.  
Therefore we have established that all the maps in the diagram 
\[
L(C,W) \leftarrow L(C',W) \to C'[\mathfrak{C}[W]^{-1}] \leftarrow 
L^S(C',\mathfrak{C}[W])\to L^S(C,W)   
\]
are DK-equivalences, which completes the proof of the proposition.  
\end{proof}    

\begin{lemma} 
\label{lem:comparison of loc 1}
Suppose $W\to C$ is a DK cofibration in $\SCat(O)$, where $W$ and $C$ are 
DK cofibrant.  Then the canonical map 
\[
L^S(C,W) \to C[W^{-1}] 
\]
is a DK equivalence.  
\end{lemma} 

\begin{proof} 
We may assume without loss of generality that $W\to C$ is a free map 
((7.4) of \cite{DK1}) between free categories ((4.5) of \cite{DK1}).  To see 
this, observe that since $W$ is DK cofibrant, it is a retract of a free category 
$U$ ((7.6) of \cite{DK1}).  Define $C' = U\cup_W C$ so that we have a 
commutative diagram 
\[
\begin{tikzcd} 
W \arrow[d] \arrow[r] & U \arrow[d] \arrow[r] & W \arrow[d] \\ 
C \arrow[r] & C' \arrow[r] & C 
\end{tikzcd} 
\]
which exhibits $W\to C$ as a retract of $U\to C'$, where $C'\to C$ is the 
canonical map.  Then $U\to C'$ is a DK-cofibration and hence is a retract 
of a free map ((7.6) of \cite{DK1}).  Hence, since $U$ is free, we may suppose without loss 
of generality that $W$ and $C$ are free, and that $W\to C$ is a free map.   
The Homotopy Lemma (6.2) of \cite{DK1} applied to the pairs $dF_*W\subset dF_*C$ and 
$W\subset C$ and the canonical map $dF_*C\to C$ then implies that the map 
$L^S(C,W)\to C[W^{-1}]$ is a DK equivalence, as required.  
%
\end{proof} 

We finish this section by describing Theorem 2.2 of \cite{DK2} dealing with certain  
localizations of projective model structures.  Following Section 2.3.2 of \cite{TV} we 
reformulate this theorem in the language of Bousfield localizations.  

Suppose that $f\colon A\to B$ is a simplicial functor between simplicial 
categories.  Recall 
(see for instance Proposition A.3.3.7 of \cite{HTT}) that the underlying functor 
$f\colon A\to B$ induces a Quillen adjunction 
\[
f_!\colon [A,\sSet]\rightleftarrows [B,\sSet]\colon f^* 
\]
for the projective model structures on $[A,\sSet]$ and $[B,\sSet]$.  Moreover, we have 
 
\begin{proposition} 
\label{prop:DK equiv induces Quillen equiv}
Let $f\colon A\to B$ be a simplicial functor between simplicial categories.  If $f\colon A\to B$ 
is a DK-equivalence, then the Quillen adjunction $(f_!,f^*)$ is a Quillen equivalence 
for the projective model structures on $[A,\sSet]$ and $[B,\sSet]$.  
\end{proposition}     

Suppose now that $f$ is a functor $f\colon (A,U)\to (B,V)$ between 
{\em pairs} of simplicial categories $(A,U)$ and $(B,V)$.  
Here by a pair of simplicial categories $(A,U)$ it is understood that 
$U\subset A$ is a subcategory of $A$ and similarly for $(B,V)$; it is also understood that 
all categories have the same sets of objects.   
Let $L_U[A,\sSet]$ denote the left Bousfield localization of the projective model 
structure on $[A,\sSet]$ with respect to the set of morphisms which is the  
image of $U$ under the Yoneda embedding; denote similarly $L_V[B,\sSet]$.   
Since $f(U)\subset V$ it follows that the adjoint pair $(f_!,f^*)$ descends 
to a Quillen adjunction 
\[
f_!\colon L_{U}[A,\sSet]\rightleftarrows L_{V}[B,\sSet]\colon f^* 
\]
between the localized projective model structures.  
A fibrant restricted diagram $F\colon A\to \sSet$ in the sense of \cite{DK2} is then 
precisely a $U$-local object in $[A,\sSet]$.  We then have the following reformulation 
of Theorem 2.2 from \cite{DK2}.    

\begin{theorem}[\cite{DK2,TV}] 
\label{thm:DK2 Thm 2.2}
Suppose that $f\colon (A,U)\to (B,V)$ is a simplicial functor between pairs 
of simplicial categories.  If $Lf\colon L^S(A,U)\to L^S(B,V)$ is a DK-equivalence then the induced adjunction 
\[
f_!\colon L_U[A,\sSet]\rightleftarrows L_V[B,\sSet]\colon f^* 
\]
is a Quillen equivalence. 
\end{theorem} 

\subsection{Localization of quasi-categories} 
\label{subsec:qcat loc}
In this subsection we study an analog of Dwyer-Kan simplicial localization for quasi-categories.  
The following definition is due to Joyal (under the name {\em homotopy localization} or {\em quasi-localization}--- see 
page 168 of \cite{Joyal-Barcelona}) and Lurie (see Definition 1.3.4.1 of \cite{HA}); further discussion of this 
concept of localization appears in \cite{Hin}.    

\begin{definition} 
\label{def:quasi-localzn}
Let $X$ be a simplicial set and let $S$ be a set of arrows in $X$.  A map $X\to Y$ in $\sSet$  
is said to exhibit $Y$ as a {\em localization} of $X$ with respect to $S$ if it satisfies the 
following universal property: for any quasi-category $Z$ the induced map $Z^Y\to Z^X$ 
is fully faithful and the essential image consists of all functors $X\to Z$ which send every 
map in $S$ to an equivalence in $Z$.  
\end{definition} 

This universal property determines $Y$ up to isomorphism in the homotopy category $\ho(\sSet)$ for the  
Joyal model structure.  We can 
take as a model for a localization $Y$ the simplicial 
set $L(X,S)$ defined by the pushout diagram 
\[
\begin{tikzcd} 
\bigsqcup_{s\in S} I \arrow[r] \arrow[d] & X \arrow[d] \\ 
\bigsqcup_{s\in S} J \arrow[r] & L(X,S).  
\end{tikzcd} 
\]
To see that $L(X,S)$ is a model for $Y$ we argue as follows.  Let $Z$ be a quasi-category. 
It is clear that the essential image of the map $Z^{L(X,S)}\to Z^X$ consists of all functors 
$X\to Z$ which send every map in $S$ to an equivalence in $Z$.  To show that the 
map $Z^{L(X,S)}\to Z^X$ is fully faithful it suffices 
to show that the diagram 
\[
\begin{tikzcd}
Z^{L(X,S)\times I} \arrow[d] \arrow[r] & Z^{X\times I} \arrow[d] \\ 
Z^{L(X,S)\times \partial I} \arrow[r]  & Z^{X\times \partial I} 
\end{tikzcd} 
\]
is a homotopy pullback for the Joyal model structure on $\sSet$, where 
the vertical maps are induced by the canonical map $\partial I\to I$.       
Therefore, since the Joyal model structure is cartesian, 
it suffices to show that the diagram 
\[
\begin{tikzcd} 
X\times \partial I \arrow[r] \arrow[d] & X\times I \arrow[d] \\ 
L(X,S)\times \partial I \arrow[r] & L(X,S)\times I 
\end{tikzcd}  
\]
is a homotopy pushout for the Joyal model structure.  A straightforward argument 
reduces this to the problem of proving that the 
square 
\[
\begin{tikzcd} 
I\times \partial I \arrow[r] \arrow[d] & I\times I \arrow[d] \\ 
J\times \partial I \arrow[r] & J\times I 
\end{tikzcd} 
\]
is a homotopy pushout in the Joyal model structure, which is clear.  

Note that the construction $L(X,S)$ is functorial in the pair $(X,S)$; hence we have a functor 
$L\colon \sSet^+\to \sSet$, where $\sSet^+$ denotes the category of {\em marked} simplicial sets 
(Section 3.1 of \cite{HTT}).  In fact a localization of $X$ with respect to $S$ may be 
represented by a fibrant replacement of $(X,S)$ in the model structure of marked 
simplicial sets (see Remark 1.3.4.2 of \cite{HA}).  Note also that there is a canonical isomorphism 
of simplicial sets $L(X,S)^{\op}\simeq L(X^{\op},S^{\op})$.    

We have the following obvious result.  

\begin{lemma} 
\label{lem:comparison of L(X,W)}
Let $X$ be a simplicial set and let $S\subset X_1$ be a set of arrows.  Then there is an 
isomorphism $\mathfrak{C}[L(X,S)] = L(\mathfrak{C}[X],S)$, where 
$L(\mathfrak{C}[X],S)$ is the simplicial category from Definition~\ref{def:luries simp loc} above.
\end{lemma}

Suppose that $u\colon A\to B$ exhibits $B$ as a localization of $A$ at a set of 
arrows $S\subset A_1$.  Let $\cS$ denote the quasi-category of spaces 
(Definition 1.2.16.1 of \cite{HTT}).   A map $A\to \cS$ corresponding to a left fibration 
$X\in \bL(A)$ via Theorem 2.2.1.2 of \cite{HTT} lies in the essential image of $u^*\colon \cS^B\to \cS^A$ if and only 
if the induced map $f_!\colon X_a\to X_b$ (see Lemma 2.1.1.4 of \cite{HTT}) 
is a homotopy equivalence for all $f\colon a\to b$ in $S$ (here $X_a$ and $X_b$ 
denote the fibers of $X$ over $a$ and $b$ respectively).  This observation 
motivates the following definition.      

\begin{definition} 
Let $A$ be a simplicial set and let $S\subset A_1$ be a set of arrows 
in $A$.  A left fibration $X\in \bL(A)$ is said to 
be {\em $S$-local} if the induced map $f_!\colon X_a\to X_b$ is a homotopy equivalence 
for all $f\colon a\to b$ in $S$. 
\end{definition} 
  
Recall (see Lemma 2.1.1.4 of \cite{HTT}) that if $X\in \bL(A)$ then 
the map $f_!\colon X_a\to X_b$ is defined up to equivalence by choosing 
a section $s$ of the trivial Kan fibration $\map_A(I,X)\to \map_A(\set{0},X)$ and defining 
$f_!$ to be the composite map 
\[
\map_A(\set{0},X)\xrightarrow{s} \map_A(I,X)\to \map_A(\set{1},X).   
\]
Therefore the left fibration $X$ is $S$-local if and only if the Kan fibration  
$\map_A(I,X)\to \map_A(\set{1},X)$ induced by the inclusion 
\[
\begin{tikzcd} 
\set{1}\arrow[dr,"b"'] \arrow[rr,hook] & &  I \arrow[dl,"f"] \\ 
& A 
\end{tikzcd} 
\]
in $\sSet/A$ is trivial for all $f\in S$.  Hence we have the following characterization 
of $S$-local left fibrations: $X\in \bL(A)$ is $S$-local if and only if 
it is a fibrant object in the left Bousfield localization $L_S\sSet/A$ of the covariant model structure $(\sSet/A,\bL(A))$  
at the set $S$ of maps $(\set{1},b)\to (I,f)$ in $\sSet/A$ 
with $f\colon a\to b$ an arrow belonging to the set $S\subset A_1$.

With this description of $S$-local left fibrations in hand, we shall prove the following 
result relating the localization $L_S\sSet/A$ of the covariant model structure on 
$\sSet/A$ and the covariant model structure on the quasi-localization $L(S,A)$.  

\begin{proposition} 
Let $A$ be a simplicial set and let $S\subset A_1$ be a set of arrows in $S$.  Let $v\colon 
A\to L(A,S)$ be the canonical map.  Then the Quillen adjunction $(v_!,v^*)$ for the 
covariant model structures on $\sSet/A$ and $\sSet/L(A,S)$ descends to a Quillen adjunction 
\[
v_!\colon L_S \sSet/A\rightleftarrows \sSet/L(A,S)\colon v^* 
\]
where $L_S\sSet/A$ denotes the left Bousfield localization as described above.  
Moreover this Quillen adjunction is a Quillen equivalence.  
\end{proposition} 

\begin{proof}
Let $v\colon A\to L(A,S)$ denote the canonical map.  
Suppose that $X\to A$ is an $S$-local left fibration.  Proposition 2.1.3.1 in \cite{HTT} 
implies that the pullback $(S\times I)\times_A X\to S\times I$ is a Kan fibration.  
For ease of notation let us write $X' = (S\times I)\times_A X$.  We may factor the 
Kan fibration $X'\to S\times I$ as $X' \to X''\to S\times I$ where $X'\to X''$ is a 
trivial Kan fibration and $X''\to S\times I$ is a minimal Kan fibration.  It is a 
classical result (paragraph 5.4 of \cite{GZ}) that $X''\to S\times I$ is trivial, that is, there is an isomorphism $X''\simeq S\times I\times M$ 
for some (minimal) Kan complex $M$.  It follows easily that there is a minimal Kan fibration 
$Z'\to S\times J$ (explicitly we may take $Z' = S\times J\times M$) and a pullback diagram 
\[
\begin{tikzcd} 
X'' \arrow[r,"\phi"] \arrow[d] & Z' \arrow[d] \\ 
S\times I \arrow[r] & S\times J 
\end{tikzcd} 
\]
An argument due to Joyal (see Lemma 2.2.5 of \cite{KLV}) shows that we may find a trivial Kan fibration 
$Z\to Z'$ and an isomorphism $X' \simeq \phi^*Z$ over $X''$.  Thus we have a commutative diagram  
\[
\begin{tikzcd} 
X \arrow[d] & \arrow[l]  (S\times I)\times_A X \arrow[r] \arrow[d] & Z \arrow[d] \\ 
A & \arrow[l] S\times I \arrow[r] & S\times J
\end{tikzcd} 
\]
in which the right hand square is a pullback and the map $Z\to S\times J$ is a Kan fibration.  Define 
\[
Y:= X\cup_{(S\times I)\times_A X}Z.  
\]   
Then the canonical map $Y\to L(A,S)$ is a left fibration since its pullback 
under the surjective map $(S\times J)\sqcup A\to L(A,S)$ is the left fibration 
\[
Z\sqcup X\to (S\times J)\sqcup A.  
\]
We have $X\simeq v^*Y$ and so we obtain another characterization of $S$-local 
fibrations: $X\in \bL(A)$ is $S$-local if and only if there exists $Y\in \bL(L(A,S))$ 
and a covariant equivalence $X\to v^*Y$.  It follows that a map 
$X\to Y$ in $\sSet/A$ is an $S$-local equivalence if and only if 
$v_!X\to v_!Y$ is a covariant equivalence in $\sSet/L(A,S)$.  Therefore we have shown that the 
Quillen adjunction $v_!\colon \sSet/A\rightleftarrows \sSet/L(A,S)\colon v^*$ 
descends to a Quillen adjunction $v_!\colon L_S\sSet/A\rightleftarrows \sSet/L(A,S)\colon v^*$.  
It is straightforward to show that 
$(v^*)^R$ is fully faithful, from which we see that this last 
Quillen adjunction is a Quillen equivalence.    
\end{proof}

If $B$ is a simplicial set then an arrow $f$ in $B$ induces an arrow 
in $\tau_1(B)$ in a canonical way.  We will say that $f$ is an {\em equivalence} in 
$B$ if its image in $\tau_1(B)$ is an isomorphism.  

\begin{lemma} 
\label{lem:descent of cov model str}
Suppose that $u\colon A\to B$ is a map of simplicial sets which sends 
every arrow in $S\subset A_1$ to an equivalence in $B$.  Then 
the Quillen adjunction $(u_!,u^*)$ between the respective covariant model structures 
descends to define a Quillen adjunction 
\[
u_!\colon L_S \sSet/A\rightleftarrows \sSet/B\colon u^*.
\]
\end{lemma} 

\begin{proof} 
We must show that every arrow $(\set{1},b)\to (I,f)$ in $\sSet/A$ is mapped 
to a covariant equivalence in $\sSet/B$.  If $B$ is a quasi-category then 
this is clear, for then $uf\colon I\to B$ factors through $J$ and $(\set{1},u(b))\to 
(J,uf)$ is a covariant equivalence in $\sSet/B$.  In general, we may compose 
with an inner anodyne $i\colon B\to B'$, where $B'$ is a quasi-category.  The 
argument just given shows that $i_!u_!$ sends every arrow in $S$ to a covariant 
equivalence in $\sSet/B'$.  But we have seen that $i_!$ reflects covariant equivalences 
(Theorem~\ref{thm:weak cat equivs and base change}).  
\end{proof} 

As an instance of this lemma, suppose that $u\colon A\to B$ is a map of simplicial sets which factors 
as 
\[
A\xrightarrow{v} L(A,S)\xrightarrow{w} B.  
\]
Then clearly every arrow in $S$ is mapped to an equivalence in $B$ and so Lemma~\ref{lem:descent of cov model str} 
implies that the Quillen adjunction $(u_!,u^*)$ between the covariant model 
structures descends to a Quillen adjunction $u_!\colon L_S \sSet/A\rightleftarrows 
\sSet/B$ from the localization $L_S\sSet/A$.  


We have the following theorem.    

\begin{theorem} 
\label{thm:loc and loc}
Let $A$ be a simplicial set and let $S\subset A_1$ be a set of arrows in $A$. 
Suppose that $u\colon A\to B$ is a map of simplicial sets which factors 
as 
\[
A\xrightarrow{v} L(A,S) \xrightarrow{w} B.  
\]
Then the Quillen adjunction 
\[
u_!\colon L_S\sSet/A\rightleftarrows \sSet/B \colon u^* 
\]
is a Quillen equivalence if and only if $w\colon L(A,S)\to B$ is a 
weak $r$-equivalence.       
\end{theorem} 

\begin{proof} 
We have proven above that $v_!\colon L_S\sSet/A\rightleftarrows \sSet/L(A,S)\colon v^*$ is a 
Quillen equivalence.  Therefore $u_!\colon L_S\sSet/A\rightleftarrows \sSet/B\colon u^*$ is a 
Quillen equivalence if and only if $w_!\colon \sSet/L(A,S)\rightleftarrows \sSet/B\colon w^*$ 
is a Quillen equivalence, if and only if $w\colon L(A,S)\to B$ is a weak $r$-equivalence 
(Proposition~\ref{prop:char of weak cat equivs}).  

\end{proof} 

\begin{remark} 
\label{rem:surjectivity}
It follows that if in addition to the hypotheses of Theorem~\ref{thm:loc and loc} above, 
$u$ is essentially surjective, then $u\colon A\to B$ 
exhibits $B$ as a localization of $A$ at $S$ if and only if the Quillen adjunction in the statement of Theorem~\ref{thm:loc and loc}  
is a Quillen equivalence.  
\end{remark}

\begin{remark} 
\label{rem:obtaining lv map}
Let $B$ be a simplicial set and let $p_B\colon N(\Delta/B)\to B$ be the last vertex map.  Let $S$ denote the set of  
final vertex maps $n\colon \Delta[0]\to \Delta[n]$ in $N(\Delta/B)$.  Then $p_B$ sends each map 
in $S$ to an identity arrow in $B$.  It follows that the composite $\bigsqcup_{s\in S}I \to N(\Delta/B)\xrightarrow{p_B} B$ 
factors through $\bigsqcup_{s\in S}\Delta[0]$ and hence through $\bigsqcup_{s\in S}J$.  Therefore there is a canonical 
map $L(N(\Delta/B),S)\to B$ and the map $p_B$ factors as the composite $N(\Delta/B)\to L(N(\Delta/B),S)\to B$.    
\end{remark}

\subsection{The delocalization theorem}
\label{subsec:deloc}
Recall from the introduction the definition of the 
last vertex map $p_B\colon N(\Delta/B)\to B$.  
Recall also that we write $S$ for the 
set of final vertex maps in $\Delta/B$.  When we want to 
emphasize the simplicial set $B$ we will sometimes write 
$S_B$ for $S$.  

Our next goal is the following key theorem from the introduction.   

\flatteningthm* 

Here the map $L(N(\Delta/B),S)\to B$ in the statement of the theorem is obtained as 
in Remark~\ref{rem:obtaining lv map}.  
As stated in the introduction, we view this theorem as an analog for 
simplicial sets of the delocalization theorem of \cite{DK2} (Theorem 2.5 of 
that paper).  


\begin{proof}
It will be slightly more convenient to replace the set $S$ with the 
set of arrows of the wide subcategory $W_B$ of $\Delta/B$ consisting of 
those arrows $u\colon \Delta[m]\to \Delta[n]$ in $\Delta/B$ for which 
$u(m) = n$.  It is straightforward to check that the canonical map 
$L(N(\Delta/B),S)\to L(N(\Delta/B),W_B)$ is a weak categorical equivalence.  

Let $F\colon \sSet\to \sSet$ denote the functor which sends a simplicial set 
$B$ to $L(N(\Delta/B),W_B)$.  Then $F$ is cocontinuous, since   
$N(\Delta/B)$ and $W_B$ are cocontinuous functors of $B$ (a proof 
of the first statement may be found in \cite{Latch} and the proof of 
the second statement is similar).
Observe that if $B\subset B'$ then $\Delta/B$ is a subcategory of 
$\Delta/B'$; it follows that $F(B)\subset F(B')$.  Since the Joyal model structure on $\sSet$ 
is left proper, a standard argument using the skeletal filtration (see for example \cite{JT3}) shows 
that we may reduce to the case in which $B = \Delta[n]$ is a simplex.    

Thus our problem is to show that $L(N(\Delta/[n]),W_{\Delta[n]})\to \Delta[n]$ is a weak categorical 
equivalence.  Equivalently, we may prove that the map $p_{\Delta[n]}\colon N(\Delta/[n])\to \Delta[n]$ is a localization 
with respect to $W_{\Delta[n]}\subset N(\Delta/[n])$.  Let $Z$ be a quasi-category, and write 
$\map_W(N(\Delta/[n]),Z)$ for the full subcategory of $\map(N(\Delta/[n]),Z)$ spanned by the 
maps $\phi\colon N(\Delta/[n])\to Z$ which send arrows in $W = W_{\Delta[n]}$ to equivalences 
in $Z$.

There is a functor 
$r\colon [n]\to (\Delta/[n])$, defined on objects by the formula 
\[
r(i) = ([i],[i]\hookrightarrow [n]), 
\]
where the map $[i]\hookrightarrow [n]$ is the canonical one corresponding to the 
inclusion $\set{0,\ldots, i}\subset \set{0,\ldots, n}$.  

Clearly we have $pr = \mathrm{id}$ (here we have written $p := p_{\Delta[n]}$).  There is a natural transformation 
$\alpha\colon \mathrm{id} \to rp$ whose components are given by 
\[
([m],\alpha\colon [m]\to [n]) \to ([\alpha(m)],[\alpha(m)] \hookrightarrow [n]) 
\]
where the map $[m]\to [\alpha(m)]$ is the canonical map induced by $\alpha$.  
Clearly the components of $\alpha$ belong to $W_{\Delta[n]}$. 

The functors $p$ and $r$ induce maps 
\[
p^*\colon \map(\Delta[n],Z)\to \map_W(N(\Delta/[n]),Z)
\]
and
\[
r^*\colon \map_W(N(\Delta/[n]),Z)\to \map(\Delta[n],Z)  
\]
respectively.  It suffices to prove that there is a natural equivalence between $p^*r^*$ 
and the identity mapping on $\map_W(N(\Delta/[n]),Z)$.  

Therefore, we may consider 
the following general situation.  Suppose that $A$ is a simplicial set and $W\subset A_1$ is a 
collection of arrows in $A$.  Suppose that $\lambda\colon A\times I\to A$ is a homotopy 
between maps $f,g\colon A\to A$ such that $f(W)\subset W$ and $g(W)\subset W$, so 
that $f$ and $g$ induce maps $f^*,g^*\colon \map_W(A,Z)\to \map_W(A,Z)$ with the 
obvious notation.  We claim that if the components $\lambda_a\colon f(a)\to g(a)$ belong to $W$, then 
$\lambda$ induces a natural equivalence $\mu\colon \map_W(A,Z)\times J\to \map_W(A,Z)$ 
from $f^*$ to $g^*$.  This is enough to prove the statement above, i.e.\ the case where 
$f = rp$ and $g = \mathrm{id}$.    

The homotopy $\lambda$ induces a homotopy $\map_W(A,Z)\times I\to \map(A,Z)$.  Since 
$\map_W(A,Z)\subset \map(A,Z)$ is a full subcategory it follows by the assumption on $f$ and 
$g$ that this homotopy restricts to a homotopy $\mu\colon \map_W(A,Z)\times I\to \map_W(A,Z)$.  To show that this 
latter homotopy $\mu$ is a natural equivalence, it suffices by Theorem 5.14 of \cite{Joyal-Barcelona} 
to prove that its components are equivalences.  Let $\phi\colon A\to Z$ be a 0-simplex of 
$\map_W(A,Z)$.  Then $\mu\colon \set{\phi}\times I\to \map_W(A,Z)$ is an edge in 
$\map_W(A,Z)$ from $\phi f$ to $\phi g$.  Since $\map_W(A,Z)\subset \map(A,Z)$ is a full subcategory, 
it suffices to show that $\mu\colon \set{\phi}\times I\to \map(A,Z)$ is an equivalence.  
Therefore, by Theorem 5.14 of op.\ cit.\ again, it suffices to show that the components 
$\mu_a\colon \phi(f(a))\to \phi(g(a))$ are equivalences in $Z$.  But $\lambda_a\colon f(a)\to g(a)$ 
belongs to $W$ by hypothesis, and hence $\mu_a$ is an equivalence for all $a\in A$, which 
completes the proof.       
\end{proof}

\begin{remark}
\label{rem:dual flattening thm}
Note that there is a canonical isomorphism $\Delta/B = \Delta/B^{\op}$ between 
the simplex categories of the simplicial set and the opposite simplicial set 
$B^{\op}$.  Under this isomorphism, a simplex $\sigma\colon \Delta[n]\to B$ is sent 
to the corresponding simplex of the opposite simplicial set $B^{\op}$; observe also 
that final vertex maps in $\Delta/B$ are sent to initial 
vertex maps in $\Delta/B^{\op}$.  Composing this 
isomorphism with $p_{B^{\op}}$, we obtain a map of simplicial sets $N(\Delta/B)\to B^{\op}$.  
The opposite $q_B:= (p_{B^{\op}})^{\op}$ is then a map  
\[
q_B\colon (N(\Delta/B))^{\op}\to B. 
\]
As for $p_B$, the map $q_B$ is determined by the functor $q_{\Delta[n]}$ 
which sends $u\colon [m]\to [n]$ in $(\Delta/[n])^{\op}$ to $u(0)$.  
Straightforward manipulations with opposites then give the dual 
version of Theorem~\ref{thm:Joyals flattening thm}:   
the map 
\[
q_B\colon N(\Delta/B)^{\op}\to B 
\]
exhibits $B$ as a localization of $N(\Delta/B)^{\op}$ at the set of initial vertex 
maps in $N(\Delta/B)$.  
\end{remark}
   
Combining the previous results with Lemma~\ref{lem:comparison of L(X,W)}, 
we obtain the following corollary.  

\begin{corollary} 
\label{corr:flattening} 
Let $B$ be a simplicial set and let $S\subset \Delta/B$ denote the set 
of final vertex maps.  Then there is an isomorphism 
\[
\mathfrak{C}[B] = L(\mathfrak{C}[N(\Delta/B)],S)
\]
in $\ho(\SCat)$.  
\end{corollary} 

There is a similar isomorphism with $S$ replaced by the set of initial vertex maps.  
From this corollary, together with Proposition~\ref{prop:Riehl} 
and Proposition~\ref{prop:DK simp loc and Lurie simp loc}, one may prove the following proposition, 
which gives a new model for the simplicial category $\mathfrak{C}[B]$.   

\begin{proposition} 
\label{thm: comparison C[B] and deltaB}
Let $B$ be a simplicial set and let $W\subset \Delta/B$ denote the 
wide subcategory of final vertex maps.  Then there is an isomorphism  
\[
 \mathfrak{C}[B] = L^H(\Delta/B,W) 
\]
in the homotopy category $\ho(\SCat)$.  
\end{proposition} 

Again, there is a similar isomorphism with $W$ replaced by 
the wide subcategory of initial vertex maps.  We leave the details to the reader.

\subsection{$L$-cofinal functors} 
\label{subsec:Lcofinal}

In this subsection we generalize some results of Dwyer and Kan on the concept of $L$-cofinal functors.  
This concept was originally introduced in the paper \cite{DK4} in the context 
of functors between ordinary categories;  
it has an evident generalization to the context of maps between simplicial sets which we now describe.

\begin{definition} 
If $u\colon A\to B$ is a map of simplicial sets then we say $u$ is {\em $L$-cofinal} 
if the following two conditions are satisfied: 
\begin{enumerate} 
\item the fiber $u^{-1}(b)$ is weakly contractible for every vertex $b\in B$; 
\item if $b\in B$ is a vertex and $1\to Rb\to B$ is a factorization of $b\colon 1\to B$ 
into a right anodyne map followed by a right fibration, then the map 
$u^{-1}(b)\to A\times_B Rb$ is left cofinal.  
\end{enumerate}
\end{definition} 

Here the map $u^{-1}(b)\to A\times_B Rb$ in (2) above is the canonical map into 
the pullback arising from the commutative diagram 
\[
\begin{tikzcd}
u^{-1}(b) \arrow[d] \arrow[r] & Rb \arrow[d] \\ 
A \arrow[r,"u"'] & B 
\end{tikzcd} 
\]
in which $u^{-1}(b)\to Rb$ is the map given by the composite 
$u^{-1}(b)\to 1 \to Rb$.    

Notice that if (2) above holds for one such factorization of $b\colon 1\to B$, then 
it holds for every such factorization.  
The following lemma gives a simpler, equivalent, formulation of the notion of 
$L$-cofinal map when $A$ and $B$ are quasi-categories.  

\begin{lemma}
Suppose that $u\colon A\to B$ is a 
map of quasi-categories.   Then $u$ is {\em $L$-cofinal} if and only if the 
following two conditions are satisfied: 
\begin{enumerate} 
\item the fiber $u^{-1}(b)$ is weakly contractible for every vertex $b\in B$; 
\item the canonical map $u^{-1}(b)\to A\times_B B_{/b}$ 
is left cofinal for every vertex $b\in B$.  
\end{enumerate} 
\end{lemma} 

\begin{proof} 
The proof of this lemma is straightforward and is left to the reader.  
\end{proof}
  
It is not hard to show that every $L$-cofinal map is both left cofinal and right cofinal 
(in fact every $L$-cofinal map is {\em dominant} --- see \cite{Joyal-Barcelona}, page 173).

Let $u\colon A\to B$ be $L$-cofinal, and let $S$ denote the set of arrows 
$S = A_1\times_{B_1}B_0$ in $A$.  Write $S$ also for the set of maps  
$(\set{1},b)\to (I,f)$ in $\sSet/A$ 
with $f\colon a\to b\in S\subset A_1$.  Then $u_!\colon \sSet/A\to \sSet/B$ 
sends every map in $S$ to a covariant equivalence in $\sSet/B$ and hence 
descends to a left Quillen functor $u_!\colon L_S\sSet/A\to \sSet/B$ from the left 
Bousfield localization of the covariant model structure on $\sSet/A$.  
We have the following result.  

\begin{theorem}
\label{thm:localization thm1} 
Let $u\colon A\to B$ be an $L$-cofinal map between simplicial sets 
and let $S = A_1\times_{B_1}B_0$.  Then the 
map $u\colon A\to B$ induces a Quillen equivalence 
\[
u_!\colon L_S\sSet/A\rightleftarrows \sSet/B\colon u^*
\]
where $L_S\sSet/A$ denotes the left Bousfield localization of the covariant model structure 
on $\sSet/A$ and where $\sSet/B$ is equipped with the covariant model structure.  
\end{theorem} 

\begin{proof} 
We show that the left derived functor $(u_!)^L$ is fully faithful and the right derived functor 
$(u^*)^R$ is conservative.  For the second statement, suppose that $X\to Y$ is a map 
in $\bL(B)$ such that $u^*X\to u^*Y$ is a weak equivalence in the localized model structure.  Then 
$u^*X\to u^*Y$ is a covariant equivalence in $\bL(A)$ and hence is a fiberwise weak homotopy 
equivalence.  Therefore $X(u(a))\to Y(u(a))$ is a weak homotopy equivalence for all $a\in A_0$, where 
$X(u(a))$ and $Y(u(a))$ denote the fibers of $X$ and $Y$ over $u(a)$.  Since $u\colon A\to B$ 
is surjective it follows that $X\to Y$ is a covariant equivalence.  

To show that the left derived functor $(u_!)^L$ is fully faithful we must show that the canonical map 
$X\to u^*Ru_!X$ is an $S$-local equivalence for every $S$-local left fibration 
$X$ on $A$, where $Ru_!X$ denotes a fibrant replacement of $u_!X$.  Since 
$u^*Ru_!X$ is $S$-local, we see therefore that $(u_!)^L$ is fully faithful 
if and only if $a^*X\to a^*u^*Ru_!X = b^*Ru_!X$ is a weak homotopy equivalence 
for every vertex $a\in A$, where $a\colon 1\to A$ and $b=u(a)\colon 1\to B$ denote 
the canonical maps.  

The map $b\colon 1\to B$ factors as $1_b\colon 1\to Rb\to B$ where $1_b\colon 1\to Rb$ is 
right anodyne and $p\colon Rb\to B$ is a right fibration.  Consider the pullback 
diagram 
\[
\begin{tikzcd}
A\times_B Rb \arrow[r,"v"] \arrow[d,"q"'] & Rb \arrow[d,"p"] \\ 
A\arrow[r,"u"'] & B 
\end{tikzcd}
\]
Then the vertical maps are right fibrations and hence are smooth in the sense of 
Definition 11.1 of \cite{Joyal-Barcelona}.  It follows (see Proposition 11.6 of \cite{Joyal-Barcelona}) that the 
diagram of functors 
\[
\begin{tikzcd} 
\ho(\sSet/A\times_B Rb) \arrow[r,"(v_!)^L"] & \ho(\sSet/Rb) \\ 
\ho(\sSet/A) \arrow[u,"q^*"] \arrow[r,"(u_!)^L"'] & \ho(\sSet/B) \arrow[u,"p^*"'] 
\end{tikzcd} 
\]
commutes up to a natural isomorphism.  In particular we have a covariant equivalence  
\[
Rv_!q^*X\to p^*Ru_!X  
\]
for any left fibration $X$ on $A$.  
Since $1_b\colon 1\to Rb$ is right anodyne, the induced map $b^*Ru_!X = 1_b^*p^*Ru_!X\to p^*Ru_!X$ 
is right anodyne and hence is a weak homotopy equivalence.  Observe that the map 
$a\colon 1\to A$ factors as the composite map $q(a,1_b)$ where $(a,1_b)\colon 1\to A\times_B Rb$ 
denotes the inclusion of the corresponding vertex.  Therefore there is a canonical map 
$a^*X\to q^*X$.  We have a commutative diagram 
\[
\begin{tikzcd} 
a^*X \arrow[rr] \arrow[d] & & b^*Ru_!X \arrow[d] \\ 
q^*X \arrow[r] & Rv_!q^*X \arrow[r] & p^*Ru_!X 
\end{tikzcd} 
\]
in which the map $q^*X\to Rv_!q^*X$ is a covariant equivalence in $\sSet/Rb$ 
and hence is a weak homotopy equivalence.  It follows therefore that the map $a^*X\to b^*Ru_!X$ is a weak 
homotopy equivalence if and only if the map $a^*X\to q^*X$ is a weak 
homotopy equivalence.  

Since $i\colon u^{-1}(b)\to A\times_BRb$ is left cofinal, the canonical map $Y\to q^*X$ is a weak 
homotopy equivalence, where $Y$ is defined by the pullback diagram 
\[
\begin{tikzcd} 
Y \arrow[r] \arrow[d] & q^*X \arrow[d] \\ 
u^{-1}(b) \arrow[r,"i"'] & A\times_B Rb 
\end{tikzcd} 
\]
Since $X$ is $S$-local, it follows from Proposition 2.1.3.1 of \cite{HTT} 
that $Y\to u^{-1}(b)$ is a Kan fibration.  
Hence $a^*Y\to Y$ is a weak homotopy equivalence since $u^{-1}(b)$ is weakly contractible.  This 
completes the proof of the theorem.   
\end{proof}

The following result is a direct analog of 2.7 of \cite{DK2}.  
 
\begin{theorem} 
\label{thm:localization thm}
Let $u\colon A\to B$ be an $L$-cofinal map between simplicial sets.  
Then the map $u\colon A\to B$ exhibits $B$ as a localization of $A$ at the set of arrows 
$S = A_1\times_{B_1}B_0$ in $A$.  
\end{theorem} 

\begin{proof}
This follows immediately from Theorem~\ref{thm:localization thm1} and 
Theorem~\ref{thm:loc and loc}.   
\end{proof}

\section{The straightening theorem} 
\label{sec:straightening}

\subsection{The straightening and unstraightening functors} 
\label{subsec:intro to straightening} 
In Section 2.2.1 of \cite{HTT} Lurie defines a pair of functors $(\St_B,\Un_B)$ --- the {\em straightening} and {\em unstraightening} 
functors respectively --- forming part of an adjunction 
\[
\St_B\colon \sSet/B\leftrightarrows 
[\mathfrak{C}[B],\sSet]\colon \Un_B.
\]  
Recall (see \cite{HTT}) that if $X\in \sSet/B$ then 
$\St_B(X)\colon \mathfrak{C}[B]\to \sSet$ is the simplicial functor defined by 
\[
\St_B(X) = \mathfrak{C}[B\cup_X 1\star X](1,-)
\]
where 1 denotes the cone point of the join $1\star X$.   
The adjunction $(\St_B,\Un_B)$ is 
a Quillen adjunction for the covariant model structure on $\sSet/B$ and the projective 
model structure on $[\mathfrak{C}[B],\sSet]$.  Recall the statement of the straightening 
theorem: 

\begin{theorem*}[Lurie \cite{HTT}]
The Quillen adjunction 
\[
\St_B\colon \sSet/B\rightleftarrows [\mathfrak{C}[B],\sSet]\colon \Un_B 
\]
is a Quillen equivalence.  
\end{theorem*}

In this section we will give a reasonably straightforward proof of this theorem.  
In fact, we will also prove a variation on the straightening theorem in which the 
directions of the left and right adjoints are reversed.  This is the `reversed' straightening 
theorem (see Theorem~\ref{thm:reversed straightening}) which we prove in the next section.  

Our strategy to prove the straightening theorem is to reduce it to the special 
case in which $B$ is equal to the nerve $NC$ of a category $C$.  
Recall (Proposition 2.2.1.1 of \cite{HTT}) that the straightening and unstraightening 
functors are natural with respect to maps of the base and hence the 
following diagram of left Quillen functors commutes up 
to a natural isomorphism: 
\[
\begin{tikzcd} 
\sSet/N(\Delta/B) \arrow[d,"{p}_!"'] \arrow[rr,"\mathrm{St}_{N(\Delta/B)}"] 
& & \left[\mathfrak{C}\left[N(\Delta/B)\right],\sSet\right] \arrow[d,"{\mathfrak{C}\left[p\right]}_!"] \\ 
\sSet/B \arrow[rr,"{\mathrm{St}}_B"'] & & \left[\mathfrak{C}\left[B\right],\sSet\right]
\end{tikzcd} 
\]  
It is easy to show, using Theorem~\ref{thm:Joyals flattening thm}, that on taking 
localizations with respect to the set of final vertex maps 
the vertical maps induce equivalences on homotopy categories.  It therefore 
suffices to prove that the top horizontal map is the left adjoint in a Quillen equivalence.  
We give the details in Sections~\ref{subsec:straightening for categories} 
and~\ref{subsec:proof of straightening thm} below.  

\subsection{A reversed straightening theorem}
\label{subsec:reversed straightening}
As an application of Theorem~\ref{thm:Joyals flattening thm}, we make the following construction.  
Let $B$ be a simplicial set.  Choose a fibrant replacement functor $\sSet/B\to \bL(B)$ for the 
covariant model structure on $\sSet/B$ (by \cite{RSS} such a functorial fibrant replacement exists because 
the covariant model structure is combinatorial).  There is a natural 
inclusion $N(\bL(B))\hookrightarrow \sNerve(\bL(B))$ from the ordinary nerve of the category underlying 
the simplicially enriched category $\bL(B)$ into the simplicial nerve of 
$\bL(B)$, and we may consider the functor $\psi\colon N(\Delta/B)\to \sNerve(\bL(B))$ defined as the composite  
\[
N(\Delta/B)\hookrightarrow N(\sSet/B)\to N(\bL(B))\hookrightarrow \sNerve(\bL(B)).  
\]
Observe that this composite functor sends every initial vertex map  
to an equivalence in $\sNerve(\bL(B))$ (but not in $N(\bL(B))$).  
Therefore, by the dual version of Theorem~\ref{thm:Joyals flattening thm} (see 
Remark~\ref{rem:dual flattening thm}), it follows that 
there is a map 
\[
\phi\colon B^{\op}\to \sNerve(\bL(B)) 
\]
such that the diagram 
\[
\begin{tikzcd} 
N(\Delta/B) \arrow[d,"p_B"'] \arrow[r,"\psi"] & \sNerve(\bL(B)) \\ 
B^{\op} \arrow[ur,"\phi"'] &
\end{tikzcd} 
\]
commutes up to an invertible 1-arrow.  In other words, there exists a map 
$h\colon N(\Delta/B)\times J\to \sNerve(\bL(B))$ restricting to $\psi$ on $N(\Delta/B)\times \set{0}$ 
and $\phi p_B$ on $N(\Delta/B)\times \set{1}$.  

By adjointness the map $\phi$ 
corresponds to a unique simplicial functor 
\[
\bar{\phi}\colon \mathfrak{C}[B]^{\op}\to \bL(B).  
\]
While this functor is difficult to describe explicitly, its action on objects 
is easy to understand; it sends a vertex $b\in B$ to the left fibration $Lb\in \bL(B)$ determined 
by a choice of a factorization $1\to Lb\to B$ of the vertex $b\colon 1\to B$ into a 
left anodyne map followed by a left fibration.    

\begin{remark}
\label{rem:left fibn map for NC}
In the case when $B=NC$ is the nerve of a category $C$, we can however give a 
much more explicit description of the functor $NC^{\op}\to \sNerve\bL(NC)$.  
In this case we may simply take the (simplicial) nerve of the functor $C^{\op}\to \bL(NC)$ 
which sends an object $c$ in $C$ to the left fibration $NC_{c/}$ over $NC$.  
If $\sigma\colon \Delta[n]\to NC$ is an object of $\Delta/NC$, then there is a  
canonical map $\Delta[n]\to NC_{\sigma(0)/}$ making the diagram 
\[
\begin{tikzcd}
\Delta[n] \arrow[rr] \arrow[dr,"\sigma"'] & & NC_{\sigma(0)/} \arrow[dl] \\ 
& NC 
\end{tikzcd}
\]
commute; this defines a natural transformation from the inclusion $\Delta/NC\subset 
\sSet/NC$ to the composite functor $\Delta/NC \to C^{\op}\to \bL(NC)\to \sSet/NC$.  
Note that the canonical map $\Delta[n]\to NC_{\sigma(0)/}$ is left anodyne.      
\end{remark}

Left Kan extension of the composite map $\mathfrak{C}[B]^{\op}\to \bL(B)\subset \sSet/B$ along 
the (simplicial) Yoneda embedding $\mathfrak{C}[B]^{\op}\to [\mathfrak{C}[B],\sSet]$ determines a 
(simplicial) adjunction 
\[
\phi_!\colon [\mathfrak{C}[B],\sSet]\rightleftarrows \sSet/B\colon \phi^!. 
\]
Here the right adjoint is the functor $\phi^!$ which on objects sends $X\in \sSet/B$ to the 
simplicial presheaf $\map_B(\bar{\phi}(-),X)$.  
Since every object in $\sSet/B$ is cofibrant, it follows easily that the right adjoint 
sends (trivial) fibrations in $\sSet/B$ to pointwise (trivial) fibrations in 
$[\mathfrak{C}[B],\sSet]$, in other words the adjunction above is a (simplicial) 
Quillen adjunction.  
Our main result in this section is the following theorem from the introduction.  

\revstr*

\begin{proof} 
From the construction of the map $\phi\colon B^{\op}\to \sNerve(\bL(B))$, 
recall that the composite functor $N(\Delta/B)\to B^{\op}\to \sNerve(\bL(B))$ is 
naturally equivalent to the functor given as the composite $N(\Delta/B)\to N(\bL(B))\hookrightarrow 
\sNerve(\bL(B))$.  By adjointness, it follows that we have two commutative diagrams 
\begin{equation} \tag{A} \label{eq:A}
	\begin{tikzcd} 
	\mathfrak{C}\left[N(\Delta/B)\right] \arrow[r] \arrow[d] & \Delta/B \arrow[d] \\ 
	\mathfrak{C}\left[N(\Delta/B)\times J\right] \arrow[r,"\bar{h}"'] & \bL(B)
	\end{tikzcd} 
\end{equation}
and 
\begin{equation} \tag{B} \label{eq:B}
	\begin{tikzcd} 
		\mathfrak{C}\left[N(\Delta/B)\times J\right] \arrow[r,"\bar{h}"] & \bL(B) \\ 
		\mathfrak{C}\left[N(\Delta/B)\right] \arrow[u] \arrow[r,"{\mathfrak{C}\left[p_B\right]}"'] & 
		\mathfrak{C}\left[B\right]^{\op} \arrow[u,"\bar{\phi}"']
	\end{tikzcd} 
\end{equation} 
Here the map $\mathfrak{C}[N(\Delta/B)]\to \Delta/B$ in diagram~\eqref{eq:A} is the counit of the adjunction $\mathfrak{C}\dashv \sNerve$ 
(recall that the simplicial nerve and the ordinary nerve coincide on $\Delta/B$), while 
the map $\Delta/B\to \bL(B)$ in diagram~\eqref{eq:A} is the composite $\Delta/B \hookrightarrow \sSet/B\to \bL(B)$ 
defining the functor $\psi$. The simplicial functor $\bar{h}$ is the adjoint of the natural 
equivalence $h\colon N(\Delta/B)\times J\to \sNerve(\bL(B))$.  

Using the inclusion $\bL(B)\subset \sSet/B$ we may extend 
diagram~\eqref{eq:A} to a commutative square in which the bottom right hand corner is $\sSet/B$.  
From this new diagram we obtain in a standard way by left Kan extension the 
following diagram of left Quillen functors 
\[
\begin{tikzcd}
\left[\mathfrak{C}\left[N(\Delta/B)\right]^{\op},\sSet\right] \arrow[r] \arrow[d] & \left[(\Delta/B)^{\op},\sSet\right] \arrow[d] \\ 
\left[\mathfrak{C}\left[N(\Delta/B)\times J\right]^{\op},\sSet\right] \arrow[r] & \sSet/B 
\end{tikzcd} 
\]
which commutes up to natural isomorphism, and in which 
$\sSet/B$ is equipped with the covariant model structure, and the categories 
of simplicial presheaves are equipped with the respective projective model structures.  
Observe that the functors $\mathfrak{C}[N(\Delta/B)]\to \Delta/B$ and $\mathfrak{C}[N(\Delta/B)] 
\to \mathfrak{C}[N(\Delta/B)\times J]$ are DK equivalences.  Therefore, by 
Proposition~\ref{prop:DK equiv induces Quillen equiv} above, 
the upper horizontal map and the left-hand vertical map are Quillen equivalences.    

The set $S\subset \Delta/B$ of initial vertex maps $0\colon \Delta[0]\to \Delta[n]$ induces, 
by Yoneda, corresponding sets of maps in $[\mathfrak{C}[N(\Delta/B)]^{\op},\sSet]$ 
and $[(\Delta/B)^{\op},\sSet]$; let us abusively denote these sets of maps by $S$ again.  
Likewise we obtain a set of maps $S\times \set{0,1}$ in 
$[\mathfrak{C}[N(\Delta/B)\times J]^{\op},\sSet]$ in the obvious way.  On passing to 
left Bousfield localizations we obtain the diagram of left Quillen 
functors 
\[
\begin{tikzcd} 
L_S\left[\mathfrak{C}\left[N(\Delta/B)\right]^{\op},\sSet\right] \arrow[r] \arrow[d] & L_S\left[(\Delta/B)^{\op},\sSet\right] \arrow[d] \\ 
L_{S\times \set{0,1}}\left[\mathfrak{C}\left[N(\Delta/B)\times J\right]^{\op},\sSet\right] \arrow[r] & \sSet/B 
\end{tikzcd} 
\]
commuting up to natural isomorphism.  
The left Quillen functor $L_S[(\Delta/B)^{\op},\sSet]\to \sSet/B$ is induced by the 
functor $R$ given as the composite $R\colon \Delta/B\to \sSet/B\to \bL(B)$, where the second 
functor is the fibrant replacement chosen above.  Let $i\colon \Delta/B\to \sSet/B$ 
denote the inclusion; there is then a natural transformation $i\to R$ which induces a 
natural transformation between the two Quillen adjunctions 
\[
[(\Delta/B)^{\op},\sSet]\rightleftarrows \sSet/B   
\]
and which becomes a natural isomorphism at the level of homotopy categories.  It follows, 
by Theorem~\ref{thm:main}, that both Quillen adjunctions are Quillen equivalences.  Hence 
the left Quillen functor 
\[
L_{S\times \set{0,1}}[\mathfrak{C}[N(\Delta/B)\times J]^{\op},\sSet]\to \sSet/B 
\]
is a Quillen equivalence.  

Now we turn our attention to the second commutative diagram~\eqref{eq:B} above.  Again, by 
composing with the inclusion $\bL(B)\subset \sSet/B$ we obtain a commutative square in which the 
top right hand corner is $\sSet/B$.  This square induces, 
in a standard way by left Kan extension and taking localizations, the following diagram 
of left Quillen functors
\[
\begin{tikzcd} 
L_{S\times \set{0,1}}\left[\mathfrak{C}\left[N(\Delta/B)\times J\right]^{\op},\sSet\right] \arrow[r] & \sSet/B \\ 
L_S\left[\mathfrak{C}\left[N(\Delta/B)\right]^{\op},\sSet\right] \arrow[r] \arrow[u] & \left[\mathfrak{C}\left[B\right],\sSet\right] \arrow[u] 
\end{tikzcd} 
\]
commuting up to natural isomorphism.
Again, the left hand vertical functor is a Quillen equivalence, and so to prove 
that the left Quillen functor $[\mathfrak{C}[B],\sSet]\to \sSet/B$ 
is a Quillen equivalence it suffices to prove that the lower right hand functor is a 
Quillen equivalence.  For this, observe that we may replace the set 
$S$ with the set of arrows in the subcategory $W$ of $\Delta/B$ consisting of the 
maps $u\colon \Delta[m]\to \Delta[n]$ in $\Delta/B$ such that 
$u(m) = n$; we then have an equality 
\[
L_S[\mathfrak{C}[N(\Delta/B)],\sSet] = 
L_W[\mathfrak{C}[N(\Delta/B)],\sSet]
\]
of model structures.  The desired statement then follows from 
a combination of Theorem~\ref{thm:Joyals flattening thm}, 
Proposition~\ref{prop:DK simp loc and Lurie simp loc}  
and Theorem~\ref{thm:DK2 Thm 2.2} above.  
\end{proof}  

\begin{remark}
\label{rem:simple straightening thm}
It follows from the proof above that if we take as the map $\phi$
the canonical functor $NC^{\op}\to \sNerve\bL(NC)$ 
described in Remark~\ref{rem:left fibn map for NC} above, we obtain 
a Quillen equivalence as in Theorem~\ref{thm:reversed straightening} above.  
Using the fact that $\mathfrak{C}[NC]\to C$ is a DK equivalence and 
Theorem~\ref{thm:DK2 Thm 2.2} we see that the Quillen adjunction 
\[
[C,\sSet]\rightleftarrows \sSet/NC 
\] 
induced by the canonical functor $C^{\op}\to \bL(NC)\subset \sSet/NC$ is 
a Quillen equivalence (or one may prove this directly).  
\end{remark}

\subsection{The straightening theorem for categories} 
\label{subsec:straightening for categories} In this section we prove 
the straightening theorem in the special case where the base is the nerve of a category.  

\begin{proposition}
Let $C$ be a small category.  Then the Quillen adjunction 
\[
\St_{NC}\colon \sSet/NC \rightleftarrows [\mathfrak{C}[NC],\sSet]\colon \Un_{NC} 
\]
is a Quillen equivalence.  
\end{proposition}

\begin{proof} 
By Theorem~\ref{thm:reversed straightening} it suffices to prove that the composite 
functor 
\[
[\mathfrak{C}[NC],\sSet] \xrightarrow{\phi_!} \sSet/NC 
\xrightarrow{\St_{NC}} [\mathfrak{C}[NC],\sSet] 
\]
is the left adjoint in a Quillen equivalence.  Since the functor 
$\phi_!\colon [\mathfrak{C}[NC],\sSet] \to \sSet/NC$ factors as 
$[\mathfrak{C}[NC],\sSet]\to [C,\sSet]\xrightarrow{u} \sSet/NC$, it suffices by 
Remark~\ref{rem:simple straightening thm} to prove that the composite 
functor 
\[
[C,\sSet]\xrightarrow{u} \sSet/NC \xrightarrow{\St_{NC}} [\mathfrak{C}[NC],\sSet]
\]
is the left adjoint in a Quillen equivalence.  Composing with the 
DK-equivalence $\psi\colon \mathfrak{C}[NC]\to C$ (see Proposition~\ref{prop:Riehl}), 
we see that it suffices (by Proposition~\ref{prop:DK equiv induces Quillen equiv} 
and Proposition 2.2.1.1 of \cite{HTT}) 
to prove that the 
composite 
\[
[C,\sSet]\xrightarrow{u} \sSet/NC\xrightarrow{\St_\psi} [C,\sSet] 
\]
is the left adjoint in a Quillen equivalence.  Let $c$ be an 
object of $C$ and let $y(c)$ be the (discrete) representable 
simplicial presheaf associated to $c$.  The image of $y(c)$ 
under $u$ is the left fibration $NC_{c/}$.  
The canonical map $1\star NC_{c/}\to NC$ sends the cone point 
to $c$; it induces a map 
\[
f\colon \St_{NC}(NC_{c/})(-)\to \mathfrak{C}[NC](c,-) 
\]  
of simplicial presheaves.  The canonical map $\mathrm{id}\colon \set{c}\to NC_{c/}$ 
induces a projective weak equivalence 
$\St_{NC}(\mathrm{id})\colon \St_{NC}(\set{c})(-)\to \St_{NC}(NC_{c/})(-)$ 
and the composite $f\circ \St_{NC}(\mathrm{id})$
is an isomorphism.  Therefore the map $f$ is a projective 
weak equivalence.  Applying the functor $\psi_!$ and composing with the 
DK equivalence $\phi$ gives a projective weak equivalence  
\[
\St_{\psi}(NC_{c/})(-)\to C(c,-) 
\]
in $[C,\sSet]$.  Moreover this map is natural in $c$.  We obtain 
therefore a natural transformation from the composite map 
$\St_{\psi}\circ u$ to the identity map on $[C,\sSet]$.   
The components of this natural transformation are projective weak 
equivalences (it suffices by results of \cite{Dugger} to check this on representables, which is the 
statement above) and therefore it follows that it defines a natural isomorphism at 
the level of homotopy categories.  Therefore, the composite 
$\St_{\psi}\circ u$ above is an equivalence at the level of 
homotopy categories, since it is naturally isomorphic to the 
identity functor.       
\end{proof}

\subsection{Proof of the straightening theorem}
\label{subsec:proof of straightening thm} 
In this section we give the details of the proof of the straightening theorem 
sketched in Section~\ref{subsec:intro to straightening} above.  
Recall that the 
following diagram of left Quillen functors commutes up 
to a natural isomorphism: 
\[
\begin{tikzcd} 
\sSet/N(\Delta/B) \arrow[d,"p_!"'] \arrow[rr,"{\St}_{N(\Delta/B)}"] 
& & \left[\mathfrak{C}\left[N(\Delta/B)\right],\sSet\right] \arrow[d,"{\mathfrak{C}\left[p\right]}_!"] \\ 
\sSet/B \arrow[rr,"{\St}_B"'] & & \left[\mathfrak{C}\left[B\right],\sSet\right] 
\end{tikzcd} 
\]
Our aim is to prove that the lower horizontal map induces an equivalence 
of homotopy categories.  

Let $S\subset \Delta/B$ denote the set of maps induced by the final vertex 
maps $n\colon [0]\to [n]$; let $S$ also denote the corresponding 
set of maps in $\mathfrak{C}[N(\Delta/B)]$.  
By Remark~\ref{rem:surjectivity} and Theorem~\ref{thm:Joyals flattening thm} the induced left 
Quillen functor $p_!\colon L_S\sSet/N(\Delta/B)\to \sSet/B$ is a left Quillen 
equivalence.  By Proposition~\ref{subsec:straightening for categories} 
above, it suffices to prove that the induced Quillen adjunction 
\[
\mathfrak{C}[p]_!\colon L_S[\mathfrak{C}[N(\Delta/B)],\sSet]\rightleftarrows 
[\mathfrak{C}[B],\sSet]\colon \mathfrak{C}[p]^* 
\]
is a Quillen equivalence.  But this has been proven above in the 
last step of the proof of Theorem~\ref{thm:reversed straightening} 
above.  

\bigskip

\noindent
{\bf Acknowledgement}: I am extremely grateful to an anonymous referee for a 
careful reading of the first version of this paper and for pointing out 
some errors and making many useful suggestions.

\end{document}